\documentclass[11pt]{amsart}
\usepackage{graphicx}
\usepackage{amsthm}
\usepackage{times,amssymb,amsmath,amsfonts}
\usepackage{xfrac}

\usepackage{hyperref}

\usepackage{color}
\usepackage{amsmath,amssymb,amsthm}
\usepackage{wrapfig}
\usepackage{ifpdf}

\newtheorem{thm}{Theorem}[section]

\newtheorem{lemma}[thm]{Lemma}

\newtheorem{definition}[thm]{Definition}
\newtheorem{conjecture}[thm]{Conjecture}
\newtheorem{remark}[thm]{Remark}

\title[$T$-avoiding spherical codes and designs]{Universal optimality of $T$-avoiding spherical codes and designs 
}

\begin{document}

\author[P. Boyvalenkov]{P. G. Boyvalenkov}
\address{ Institute of Mathematics and Informatics, Bulgarian Academy of Sciences,
8 G Bonchev Str., 1113  Sofia, Bulgaria}
\email{peter@math.bas.bg}

\author[D. Cherkashin]{D. D. Cherkashin}
\address{ Institute of Mathematics and Informatics, Bulgarian Academy of Sciences,
8 G Bonchev Str., 1113  Sofia, Bulgaria}
\email{jiocb@math.bas.bg}

\author[P. Dragnev]{P. D. Dragnev}
\address{ Department of Mathematical Sciences, Purdue University 
Fort Wayne, Fort Wayne, IN 46805, USA }
\email{dragnevp@pfw.edu}

\begin{abstract}
Given an open set $T\subset [-1,1)$, we introduce the concepts of $T$-avoiding spherical codes and designs, that is, spherical codes that have no inner products in the set $T$. We show that certain codes found in the minimal vectors of the Leech lattice, as well as the minimal vectors of the Barnes--Wall lattice and codes derived from strongly regular graphs, are universally optimal in the restricted class of $T$-avoiding codes. We also extend a result of Delsarte--Goethals--Seidel about codes with three inner products $\alpha, \beta, \gamma$ (in our terminology  $(\alpha,\beta)$-avoiding $\gamma$-codes). Parallel to the notion of tight spherical designs, we also derive that these codes are minimal (tight) $T$-avoiding spherical designs of fixed dimension and strength. In some cases, we also find that codes under consideration have maximal cardinality in their $T$-avoiding class for given dimension and minimum distance.  
\end{abstract}

\maketitle

\centerline{\it On the occasion of Ed Saff's 80th birthday}

\bigskip


\bigskip


\section{Introduction}

Let $\mathbb{S}^{n-1}=\{x=(x_1,\ldots,x_n): x_1^2+\cdots+x_n^2=1\}$ be the 
unit sphere in $n$ dimensions. A finite set $C \subset \mathbb{S}^{n-1}$
is a spherical $(n,N,s)$ code if $|C|=N$ and 
\[
s=s(C)=\max \{ x \cdot y : x,y \in C,
x \neq y\}.
\]
The parameter $s(C)$ is called the maximal cosine of $C$ and 
we shall also refer to $C$ as an $s$-code. Denote by 
\[ I (C):=\{ x \cdot y : x,y \in C, x \neq y\} \]
the set of all inner products of distinct points of $C$. Note that $s = \max I(C)$.

\begin{definition}
Let $T \subset [-1,1)$. A spherical code $C \subset \mathbb{S}^{n-1}$  
is called $T$-avoiding if $I(C) \cap T = \varnothing$.  
\end{definition}

As is often the case in the study of good spherical codes, the notion of 
spherical designs is very useful. 

\begin{definition}
A spherical $\tau$-design is a spherical 
code $C \subset \mathbb{S}^{n-1}$ such that
\begin{equation} \label{MultiPolQ}
\int_{\mathbb{S}^{n-1}} p(x) d\sigma_n(x)= \frac{1}{|C|} \sum_{x \in C} p(x) 
\end{equation}
($\sigma_n $ is the normalized surface measure) holds for all polynomials $p(x) = p(x_1,x_2,\ldots,x_n)$ of total degree at most $\tau$.
\end{definition}

The notion of energy of a code $C$ for a given interaction potential $h$ and finding energy minimizers has been the focus of many research articles (see, for example, \cite{BGMPV1,BHS2019,BDHSS2016,B2024,BG2015,CK2007,CKMRV22,DLT,HarSaf2004,KY1997,SafKui1997,Schw} and references therein).

\begin{definition}
Given a function $h:[-1,1)\to (-\infty,+\infty)$, continuous on $[-1,1)$, we consider the $h$-energy of a spherical code $C$ as 
\begin{equation} \label{h-energy} E_h (C):=\sum_{x\not= y\in C} h(x\cdot y).
\end{equation} 
\end{definition}

In their seminal paper \cite{CK2007} Cohn and Kumar introduced the concept of universally optimal codes as global minimizers (of fixed cardinality) of the $h$-energy for all absolutely monotone potentials, that is, potentials $h$ such that $h^{(j)}\geq 0$, $j\geq 1.$ In this regard, we introduce the notion of $T$-avoiding universally optimal codes.

\begin{definition}
Given an open set $T\subset [-1,1)$, a code $C\subset \mathbb{S}^{n-1}$ is called $T$-avoiding universally optimal if for any absolutely monotone potential $h$ it minimizes the $h$-energy among all $T$-avoiding codes with cardinality $|C|$. 
\end{definition}

In \cite{GV2024} Gon{\c{c}}alves--Vedana consider lattices in $\mathbb{R}^{48}$ 
which avoid certain distances that after reformulating into inner products, lead to the set $T=(-1/3,-1/6) \cup (1/6,1/3)$. The main result in \cite{GV2024}
is that among these, say $T$-avoiding lattices (with the above $T$), 
four known lattices in $\mathbb{R}^{48}$ have the maximal possible density. This is a result
in the spirit of the celebrated proofs of the optimality of the $E_8$ lattice by Viazovska \cite{V} and the Leech lattice by Cohn--Kumar--Miller--Radchenko--Viazovska \cite{CKMRV17} (see also \cite{CKMRV22}). 
Moreover, as shown in \cite{BC2025}, the spherical codes defined by the minimal vectors of the same four lattices solve the kissing number problem (with some restrictions) among $T$-avoiding (1/2)-codes. Furthermore, these codes were also shown in \cite{BD2025} to be universally optimal (in Cohn--Kumar sense) among $T$-avoiding codes (with some restrictions), meaning that for any absolutely monotone potential function $h$, the $h$-energy is minimal in the class of $T$-avoiding codes.

The results described in the last paragraph prompt us to consider, for special choices of dimension $n$, set $T$, and other relevant parameters, the following problems about maximal cardinality of $T$-avoiding $s$-codes, minimal cardinality for $T$-avoiding designs and minimum $h$-energy of $T$-avoiding codes.

{\bf Problem 1}: For given $n$, $s$, and $T$, find the maximal cardinality of $T$-avoiding $s$-codes on $\mathbb{S}^{n-1}$. Find {\em maximal $T$-avoiding $s$-codes}. 

\begin{remark} 
The classical coding theory problem seeks to find the maximal cardinality of $s$-codes on $\mathbb{S}^{n-1}$, which in our terminology is to determine the maximal cardinality of $(s,1)$-avoiding codes. Instead of adding $(s,1)$ to our sets $T$ everywhere, we adhere to the widely used notion of $s$-codes (which is equivalent). 
\end{remark}

{\bf Problem 2}: For given $n$, $\tau$, and $T$, find the minimum cardinality of $T$-avoiding designs on $\mathbb{S}^{n-1}$. Find {\em tight $T$-avoiding spherical $\tau$-designs}. 

{\bf Problem 3}: For given $n$, $N$, an absolutely monotone potential $h$, and $T$ find the minimum $h$-energy (see \eqref{h-energy} above) of $T$-avoiding codes of cardinality $N$ on $\mathbb{S}^{n-1}$. Find {\em universally optimal $T$-avoiding codes}.

In this paper we adopt and apply linear programming techniques for these problems. We focus on codes embedded in the Leech lattice, the minimal vectors of the Barnes--Wall lattice, three-distance codes (or $|I(C)|=3$), and on codes obtained as spectral embeddings of strongly regular graphs. Then we choose appropriate sets $T$ to be avoided for each code and prove that our codes are optimal for Problems 1, 2, and 3, respectively. We note however, that the technique is more general and may be applied to other spherical codes and designs. We note also that the universal optimality of the $T$-avoiding codes we discuss holds for a larger class of potentials (see Remark \ref{AbsM}).

The paper is organized as follows. In Section 2 we collect necessary definitions and facts about spherical codes and designs. The general linear programming theorems corresponding to Problems 1, 2, and 3 are shown in Section 3 together with the conditions of attaining these bounds. In Section 4 we give initial descriptions of our target codes as we list their parameters and provide information for relations between them. Section 5, 6, and 7 are devoted to maximal $T$-avoiding codes, minimum $T$-avoiding designs, and universally optimal $T$-avoiding codes, respectively. The first proof in each section contains all necessary technical details while we only sketch the remaining proofs and give the details in an Appendix.

\section{Preliminaries}

In this section, we introduce some preliminaries that will allow us to state and prove our results. For fixed dimension $n$ we consider the collection of {\em Gegenbauer orthogonal polynomials} $P_i^{(n)}$, $i=0,1,\dots$, where the measure of orthogonality is $$d\mu_n(t)=\gamma_n (1-t^2)^{(n-3)/2}\, dt$$ with a normalizing constant $\gamma_n$ chosen so that $\mu_n$ becomes a probability measure. The Gegenbauer polynomials are uniquely determined by the conditions $P_i^{(n)}(1)=1$. Note that Gegenbauer polynomials are Jacobi polynomials $P_i^{(\alpha,\beta)}$ with parameters $\alpha=\beta=(n-3)/2$. 

Given a code $C\subset \mathbb{S}^{n-1}$ we define the $i$-th moment of $C$, $i \geq 1$, as
$$ M_i(C):=\sum_{x, y \in C} P_i^{(n)}(x\cdot y).$$ 
A key property of Gegenbauer polynomials is their positive definiteness~\cite{Sch42}, which means that $M_i (C) \geq 0$ for every $i$ and $C$.

Then a code $C$ is a spherical $\tau$-design if and only if $M_i(C)=0$, $i=1,\dots,\tau$. 
Another useful property of moments is that an antipodal code $C$ (such a code that $C = -C$) has $M_i(C) = 0$ for all odd $i$.

The linear programming method on the sphere is based on building a polynomial $f$ with a proper sign on $[-1,1]\setminus T$ and proper signs of the coefficients in the Gegenbauer expansion
\begin{equation} \label{geg-exp}
f(t) = \sum_{i=0}^{\deg f} f_i P_i^{(n)}(t).
\end{equation}
The coefficients of the expansion \eqref{geg-exp} satisfy
\[
f_i = \frac{\int_{-1}^{1} f(t) \cdot P_i^{(n)}(t) \cdot d \mu_n(t)}{\int_{-1}^{1} \left(P_i^{(n)}(t) \right)^2 \cdot d \mu_n(t)}.    
\]

An $(n, N, d, \tau)$-\textit{configuration} is a spherical code with $N$ points on $\mathbb{S}^{n-1}$ which is a spherical $\tau$-design and $|I(C)| = d$ (i.e. $C$ has $d$ different positive distances). 

For any point $x \in C$ and inner product $t \in I(C)$, we denote by
\[ A_t(x):=|\{y \in C : \langle x,y \rangle  = t\}| \]
the number of the points of $C$ with inner product $t$ with
$x$. Then the system of nonnegative integers
$$(A_t(x): t \in I(C))$$ is called the \textit{distance distribution
of $C$ with respect to the point $x$}. If $A_t(x)$ does not depend on the choice of $x \in C$ for every $t \in I(C)$, then the code $C$ is called {\it distance invariant} (cf. \cite[Definition 7.2]{DGS}). We will omit $x$ in the notation for distance invariant codes. Note that $A_{-1}=1$ means that $C$ is antipodal, i.e. $C=-C$. For antipodal codes we have
$A_t(x)=A_{-t}(x)$ for every $t \in I(C) \setminus \{-1\}$ and every $x \in C$.

For a distance-invariant $C$ we consider the distance distribution of $C$ to be the ordering $F(C)$ of $(A_t(x): t \in I(C))$ as the order corresponds to the increasing of the elements of $I(C)$ (in particular, for an antipodal $C$ the set $F(C)$ starts with 1). Thus, 
\[ F(C):= \left( A_{t_i} : t_i \in I(C) \right), \]
where $I(C)=\{t_1,\ldots,t_d\}$ and $t_1<\cdots<t_d$. 

\begin{thm}[part of Theorem 7.4 in~\cite{DGS}] \label{thm:DGSdistances}
Let $C$ be an $(n, N, d, \tau)$-configuration with $\tau \geq d - 1$.
Then $C$ is distance invariant.  
\end{thm}

If the set $I(C)$ in the assumptions of Theorem~\ref{thm:DGSdistances} is known, then one can find the distance distribution $F(C)$ from the system $M_i(C)=0$, $i=1,\dots,\tau$, or as described in the end of this section.

Starting with a code $C \subset \mathbb{S}^{n-1}$, a point $x \in C$, and an inner product $\alpha \in I(C)$ one can define a derived code $C_{\alpha,x} \subset \mathbb{S}^{n-2}$ after proper rescaling of all points of $C$ having inner product $\alpha$ with $x$. When $C$ is distance-invariant, $I(C_{\alpha,x})$ does not depend on $x$ and we
omit $x$ in the notation. In this case we have (see \cite[Section 8]{DGS})
\[
I(C_{\alpha}) = \left\{\frac{\beta-\alpha^2}{1-\alpha^2} : \beta \in I(C) \right\} \cap [-1,1). 
\]

A derived code inherits some properties of the initial one as shown in the next theorem from \cite{DGS}. 

\begin{thm}[Theorem 8.2 in~\cite{DGS}] \label{thm:DGSderived}
Let $C$ be a $\tau$-design, $x \in C$ and $\alpha \in I(C) \setminus \{-1\}$. Assume that $|I(C_{\alpha,x})| \leq \tau + 1$. Then $C_{\alpha,x}$ is a $(\tau + 1 - |I(C_{\alpha,x})|)$-design.
\end{thm}

A similar construction and theorem hold for $x\not\in C$, which we state below.

\begin{thm}[Theorem 3.2 in~\cite{BBDHSS2025}]
\label{derived_codes_thm}
Let $C\subset \mathbb{S}^{n-1}$ be a spherical $\tau$-design. Suppose there is $\widetilde{x}\in\mathbb{S}^{n-1}$ and $\ell \leq \tau$ real numbers $-1<\alpha_1<\dots<\alpha_\ell<1$, such that for all $y\in C$, $y\cdot \widetilde{x} \in \{-1,\alpha_1,\dots,\alpha_\ell,1\}$. Then the derived 
codes $C_{\alpha_i,\widetilde{x}}$, $i=1,\dots,\ell$ are spherical $(\tau+1-\ell)$-designs on $\mathbb{S}^{n-2}$.
\end{thm}

If $C \subset \mathbb{S}^{n-1}$ is a distance invariant $\tau$-design, then for every $x \in \mathbb{S}^{n-1}$ and every polynomial $f$ of degree at most $\tau$ one may consider the polynomial $p(y)=f(x\cdot y)$ and apply \eqref{MultiPolQ} and the Funk--Hecke formula to conclude that
\[
N f_0 = \sum_{y \in C} f(x \cdot y)
\]
holds (see also Equation (1.10) in \cite{FL95}). When we select $x\in C$ we derive the existence of a quadrature formula with nodes $I(C)\cup\{1\}$ and weights $F(C)\cup \{1\}$: i.e. for every polynomial $f$ of degree at most $\tau$ one has 
\begin{equation} \label{QF}
N f_0 = \sum_{i=1}^{d+1} m_i f(\alpha_i), 
\end{equation}
where $m_i$ are the multiplicities from $F(C) \cup \{1\}$ and $\alpha_i$ are the corresponding nodes from $I(C) \cup \{1\}$, $\alpha_{d+1}=1$. Knowing the inner products, the quadrature rule \eqref{QF} allows us to compute the distance distribution of spherical designs of strength $\tau \geq d-1$ using the Lagrange basis induced by the nodes $I(C)\cup\{1\}$ (or using the polynomials $f(t)=1,t,\ldots,t^{d-1}$ for a Vandermonde system).

\section{General linear programming theorems}

Modifications of the classical linear programming (LP) bounds for spherical codes and designs \cite{DGS,KL} can be easily obtained. 
The proofs follow directly from the identity 
\begin{equation} \label{main-id}
f(1) \cdot |C|+\sum_{x,y \in C, x \neq y} f(\langle x,y \rangle)=
f_0 \cdot |C|^2 + \sum_{i=1}^{\deg{f}} f_iM_i(C)
\end{equation}
(here $f(t)$ has Gegenbauer expansion \eqref{geg-exp}  
and $C \subset \mathbb{S}^{n-1}$ is a code) and the corresponding conditions in the theorems. The
identity \eqref{main-id} itself is obtained by computing both 
sides of \eqref{geg-exp} in all inner products (counting the multiplicities) of points $C$ and summing.

We proceed with definitions of the target quantities for Problems 1, 2, and 3 and the corresponding linear programming theorems for each of them. 

\subsection{Maximal cardinality of \texorpdfstring{$T$}{T}-avoiding \texorpdfstring{$s$}{s}-codes}

Denote by 
\[ 
\mathcal{A}(n,s;T):= \max \left\{ |C|: C \subset \mathbb{S}^{n-1} \mbox{ is a $T$-avoiding spherical code}, \ s(C) \leq s \right\} 
\]
the maximal cardinality of a $T$-avoiding spherical code with maximal inner product at most $s$. We can assume without loss of generality that $T \subset [-1,s)$.  

\begin{thm}[LP for unrestricted codes~\cite{DGS}] \label{general-lp-max-codes}
 Let $T \subset [-1,s) $ and $f(t) \in \mathbb{R}[t]$ be such that

{\rm (A1)} $f(t) \leq 0$ for all $t \in [-1,s] \setminus T$;

{\rm (A2)} $f_0>0$ and $f_i \geq 0$ for all $i$ in the Gegenbauer expansion $f(t)=\sum_{i=0}^{\deg{f}} f_iP_i^{(n)}(t)$.

Then 
\[ |C| \leq \frac{f(1)}{f_0} \]
for any spherical code $C \subset \mathbb{S}^{n-1}$ such that $I(C) \subset ([-1,s] \setminus T)$. Consequently,
\[ \mathcal{A}(n,s;T) \leq \inf \left\{ \frac{f(1)}{f_0} : f \mbox{ satisfies (A1) and (A2)} \right\}.  \]
If a $T$-avoiding spherical code $C \subset \mathbb{S}^{n-1}$ with $s(C) \leq s$ and a polynomial $f$ satisfying (A1) and (A2) 
are such that $|C|=f(1)/f_0$, then $I(C)$ is a subset of the set of the zeros of $f$ and $f_iM_i(C)=0$ for each $i \geq 1$. In particular, if $f_i>0$ for
$i=1,2,\ldots,m$ for some $m \leq \deg{f}$, then $C$ is a ($T$-avoiding)
spherical $m$-design. 
\end{thm}

\subsection{Minimum cardinality of \texorpdfstring{$T$}{T}-avoiding designs} 

Denote by 
\[ \mathcal{B}(n,\tau;T):= \min \left\{ |C|: C \subset \mathbb{S}^{n-1} \mbox{ is a $T$-avoiding spherical $\tau$-design} \right\} \]
the minimum cardinality of a $T$-avoiding spherical $\tau$-design. 

\begin{thm}[LP for spherical designs~\cite{DGS}] \label{general-lp-des}
Let $\tau$ be a positive integer, $T \subset [-1,1)$, and $f(t) \in \mathbb{R}[t]$ be such that

{\rm (B1)} $f(t) \geq 0$ for all $t \in [-1,1] \setminus T$;

{\rm (B2)} $f_0>0$, $f_i \leq 0$ for all $i \geq \tau+1$ in the Gegenbauer expansion $f(t)=\sum_{i=0}^{\deg{f}} f_iP_i^{(n)}(t)$.

Then 
\[ |C| \geq \frac{f(1)}{f_0} \]
for any spherical $\tau$-design $C \subset \mathbb{S}^{n-1}$ such that $I(C) \subset ([-1,1] \setminus T)$. Consequently,
\begin{equation} \label{lp-des-T}
\mathcal{B}(n,\tau;T) \geq \sup \left\{ \frac{f(1)}{f_0} : f \mbox{ satisfies (B1) and (B2)} \right\}. 
\end{equation}
If a $T$-avoiding spherical $\tau$-design $C \subset \mathbb{S}^{n-1}$ and a polynomial $f$ satisfying (B1) and (B2) are such that
$|C|=f(1)/f_0$, then $f_iM_i(C)=0$ for each $i \geq \tau+1$ and
$I(C)$ is a subset of the set of the zeros of $f$.
\end{thm}

\begin{definition}
A $T$-avoiding spherical $\tau$-design $C \subset \mathbb{S}^{n-1}$ which attains the bound \eqref{lp-des-T} for some polynomial $f$ is called a tight $T$-avoiding spherical $\tau$-design.
\end{definition}

We remark that all polynomials for Theorem \ref{general-lp-des} in this paper will be of degree at most $\tau$ which means that  we have to check only $f_0>0$ in the condition (B2) and the usually easy condition (B1). 

\subsection{Minimum energy of \texorpdfstring{$T$}{T}-avoiding codes}

Denote by 
\[ \mathcal{E}_h(n,N;T):= \min \left\{ E_h(C): C \subset \mathbb{S}^{n-1} \mbox{ is a $T$-avoiding spherical code}, \ |C|=N \right\} \]
the minimum $h$-energy of a $T$-avoiding spherical code on $\mathbb{S}^{n-1}$ of cardinality $|C|=N$.

\begin{thm}[LP for energy~\cite{BD2025}]  \label{general-lp-energy} 
Let $h(t)$ be a potential function and $f \in \mathbb{R}[t]$ be such that

{\rm (E1)} $f(t) \leq h(t)$ for every $t \in [-1,1)\setminus T$;

{\rm (E2)} $f_i \geq 0$ for all $i \geq 1$. 

Then 
\[
 E_h(C) \geq N^2 \left( f_0-\frac{f(1)}{N}\right)
\]
for every $T$-avoiding spherical code $C \subset \mathbb{S}^{n-1}$ of cardinality $|C|=N$. Consequently,
\[
\mathcal{E}_h(n,N;T) \geq \sup \left\{ N^2\left( f_0-\frac{f(1)}{N}\right): f \mbox{ satisfies (E1) and (E2)} \right\}.
\]
If a $T$-avoiding spherical code $C \subset \mathbb{S}^{n-1}$ and a polynomial $f$ satisfying (E1) and (E2) are such that
$E_h(C)=N^2 \left( f_0-f(1)/N\right)$, then $f_iM_i(C)=0$ for each $i \geq 1$ and $I(C)$ is a subset of the set of the zeros of the function $h-f$.
\end{thm}

\subsection{Examples in 48 dimensions}

We present as examples recent results on $T$-avoiding (1/2)-codes in 48 dimensions from \cite{BC2025} and \cite{BD2025}. 

Let $T_1:=\left(-1/3,-1/6\right) \cup \left(1/6,1/3\right)$ and $T_2 := \left(-1/2,-1/3\right) \cup \left(1/3,1/2\right)$.

\begin{thm}[Theorem 5.1 in \cite{BC2025}] \label{kiss-48}
Let $C \subset \mathbb{S}^{47}$ be a $T_1$-avoiding spherical $(1/2)$-code with $M_3(C)=0$.
Then $|C| \leq 52\,416\,000$. If the equality is attained, then $C$ is an antipodal spherical 11-design which is distance invariant and 
its distance distribution is as given in~\eqref{eq:distdist1} and~\eqref{eq:distdist2}.
\end{thm} 

\begin{thm}[Theorems 6.1 and 6.2 in \cite{BC2025}] \label{11des-48}
Let $C \subset \mathbb{S}^{47}$ be a $T_i$-avoiding spherical $11$-design for either $i=1$ or $i=2$. Then
$|C| \geq 52\,416\,000$. If the equality is attained, then $C$ is a $(48,52\,416\,000,1/2)$ antipodal spherical code which is distance invariant 
with distance distribution as in~\eqref{eq:distdist1} and~\eqref{eq:distdist2}.
\end{thm}

There are at least four spherical 11-designs on $\mathbb{S}^{47}$, formed as the sets of minimal vectors of the even unimodular extremal lattices $P_{48p}, P_{48q}, P_{48m}$, and $P_{48n}$
in $\mathbb{R}^{48}$ (see \cite{CS,nebe2014fourth}). All these codes are distance invariant with 8 distinct distances (cf., e.g., via \cite[Theorem 7.4]{DGS}) and all have the same inner products
\begin{equation} \label{eq:distdist1}
I(C) = \left\{-1, -\frac{1}{2}, -\frac{1}{3}, -\frac{1}{6}, 0, \frac{1}{6}, \frac{1}{3}, \frac{1}{2}  \right\} 
\end{equation}
and distance distribution 
\begin{equation} \label{eq:distdist2}
F(C)=(1, 36\,848, 1\,678\,887, 12\,608\,784, 23\,766\,960, 12\,608\,784, 1\,678\,887, 36\,848). 
\end{equation}

It follows from Theorem \ref{11des-48} that these codes are tight $T_i$-avoiding spherical $11$-designs either for $i=1$ or $2$.  

The $h$-energy of each of the above four codes (say, $C$) is given by
\begin{eqnarray*} 
\frac{E_h(C)}{52416000} &=& 36848\left(h\left(-\frac{1}{2}\right)+h\left(\frac{1}{2}\right)\right) + 1678887\left(h\left(-\frac{1}{3}\right)+h\left(\frac{1}{3}\right)\right) \nonumber \\ 
&& \, +12608784\left(h\left(-\frac{1}{6}\right)+h\left(\frac{1}{6}\right)\right)+23766960h(0)+h(-1).
\end{eqnarray*}

\begin{thm} \cite{BD2025} \label{ulb-48}
Let $h$ be absolutely monotone with $h^{(12)}>0$ in $(-1,1)$. 
Let $C \subset \mathbb{S}^{47}$ of cardinality $|C|=52\,416\,000$ be either $T_1$-avoiding with $M_3(C)=0$ or $T_2$-avoiding. Then 
\begin{equation} \label{ulb-h-energy-48}
\begin{split}
\frac{E_h(C)}{52416000} &\geq 36848\left(h\left(-\frac{1}{2}\right)+h\left(\frac{1}{2}\right)\right) + 1678887\left(h\left(-\frac{1}{3}\right)+h\left(\frac{1}{3}\right)\right) \\ 
& \ \ \ \, +12608784\left(h\left(-\frac{1}{6}\right)+h\left(\frac{1}{6}\right)\right)+23766960h(0)+h(-1).
\end{split}
\end{equation}
The equality is attained when $C$ is an antipodal spherical 11-design 
which is distance invariant and 
its distance distribution is as given in~\eqref{eq:distdist1} and~\eqref{eq:distdist2}. In particular, the four codes formed by the minimum norm vectors in the  
even unimodular extremal lattices $P_{48p}, P_{48q}, P_{48m}$, and $P_{48n}$ in $\mathbb{R}^{48}$, respectively, attain the bound \eqref{ulb-h-energy-48} and hence, are universally optimal among the considered class of codes.
\end{thm}

\section{Codes}

In this section we describe good spherical codes and designs whose optimality will be shown in the next three sections. Most of these codes can be seen as derived codes in the celebrated kissing configuration $L \subset \mathbb{S}^{23}$ of $196560$ points defined by the minimum non-zero norm in the Leech lattice $\Lambda_{24}$. Recall that the Leech lattice is the unique even unimodular lattice in $\mathbb{R}^{24}$ with no roots, i.e. with minimum norm $4$. By appropriate scaling and translation it is useful to embed spherical codes as vectors in the Leech lattice that lie on certain hyperspheres. Following Conway--Sloane \cite{CS} notation we shall denote the set of $196560$ minimal vectors of length $2$ by $\Lambda(2)$ and the set of $16773120$ vectors of length $\sqrt{6}$ as $\Lambda(3)$. 

\subsection{Codes on \texorpdfstring{$\mathbb{S}^{22}$}{S22}}\label{Codes_22} We start with derived codes obtained from $L$ according to Theorems \ref{thm:DGSderived} and \ref{derived_codes_thm}. Fix a point $x \in L$ and consider all points of $L$ with inner product $\alpha \in \{1/2, 1/4, 0\}$ with $x$. After proper rescaling, we obtain codes $C_{\alpha} \subset \mathbb{S}^{22}$ with inner products 
\begin{equation} \label{I-derived}
I(C_{\alpha})=\left\{\frac{\beta-\alpha^2}{1-\alpha^2} : \beta \in I(L) \right\} \cap [-1,1).
\end{equation}
These three codes are spherical $7$-designs on $\mathbb{S}^{22}$ (therefore distance invariant). Note that $M_9(C_{\alpha})=0$ in all three $\alpha$'s; moreover, the codes $C_{1/2}$ and $C_0$ are in fact antipodal.  

The smallest one, $C_{1/2}$, is a tight spherical $7$-design on $\mathbb{S}^{22}$ of cardinality $4600$, that
is unique up to isometry \cite{BS1981} (see also \cite{CK-NYJM}), and that is universally optimal \cite{CK2007}. 

The second code, $C_{1/4}$ has cardinality $47104$, inner products 
\[ I(C_{1/4})=\left\{ -\frac{3}{5},-\frac{1}{3},-\frac{1}{15},\frac{1}{5},\frac{7}{15}\right\} \]
(according to \eqref{I-derived}), and distance distribution
\[ F(C_{1/4})=\left(275,7128,22275,15400,2025\right). \]

The third code, $C_0$, is located in the equatorial hypersphere (with respect to $x$) of $\mathbb{S}^{23}$. It has cardinality $93150$ and the same inner products as $L$ (since $\alpha=0$ in \eqref{I-derived}); i.e., 
\[ I(C_0)=\left\{ -1,-\frac{1}{2},-\frac{1}{4}, 0, \frac{1}{4},  \frac{1}{2} \right\}. \]
Its distance distribution is 
\[ F(C_0)=(1, 2464, 22528, 43164, 22528, 2464). \]

As Theorem \ref{derived_codes_thm} shows, good derived codes can be obtained not only when starting from a point of the original code. In fact, one can also use suitable centers of spherical caps defined by facets, for instance centers corresponding to deep holes. The deep holes of $\Lambda_{24}$ were classified by Niemeier \cite{Nie2} (see Chapter 16 in \cite{CS}). There are 23 inequivalent deep holes, one for each of the 23 even unimodular lattices in $\mathbb{R}^{24}$ found by Niemeier. Probably the most interesting one corresponds to a facet of the code $L$ with 552 vertices, which we describe below.

If we fix a vector $W\in \Lambda(3)$ there are $552$ vectors from $\Lambda(2)$ that are at a distance $2$ from $W$, $11178$ vectors at a distance $\sqrt{6}$ and $48600$ at a distance $\sqrt{8}$. If $\widetilde{w}$ is the projection of $W$ onto $\mathbb{S}^{23}$ (note that $L$ is the respective projection of $\Lambda(2)$), then these three codes are derived codes $C_{\sqrt{6}/4,\widetilde{w}}$, $C_{\sqrt{6}/6,\widetilde{w}}$, and $C_{\sqrt{6}/12,\widetilde{w}}$, respectively. Since the parameters under consideration (inner products and distance distribution) do not depend on the point $\widetilde{w}$ we will omit it in the sequel. 

From Theorem \ref{derived_codes_thm} all three codes are spherical $5$-designs (observe that $L$ is contained in $7$ parallel hyperplanes, i.e. $\ell = 7$). We note that the spherical design property for these subcodes may be verified directly using the coordinate representation of $\Lambda(2)$, e.g. found in \cite[Chapter 10]{CS}. 

The $552$ facet mentioned above is $C_{\sqrt{6}/4}$. It is a tight $5$-design on $\mathbb{S}^{22}$ and hence a universally optimal code (see \cite{CK2007}).

The next derived code $C_{\sqrt{6}/6}\subset \mathbb{S}^{22}$ has a cardinality $11178$ and has four distinct inner products, 
\[ I(C_{\sqrt{6}/6})=\left\{-\frac{1}{2}, -\frac{1}{5}, \frac{1}{10}, \frac{2}{5}\right\}. \]
The distance distribution of $C_{\sqrt{6}/6}$ is
\[ F(C_{\sqrt{6}/6})=(352,4125,5600,1100). \]

The third derived code we will consider is $C_{\sqrt{6}/12}$ which has cardinality $48600$ and five distinct inner products, 
\[ I(C_{\sqrt{6}/12})=\left\{-\frac{13}{23}, -\frac{7}{23}, -\frac{1}{23}, \frac{5}{23}, \frac{11}{23}\right\} \]
and distance distribution 
\[ F(C_{\sqrt{6}/12})=( 506, 8855, 23046,14421,1771). \]

We remark that $\widetilde{w}$ is the {\em universal minimum} of the $h$-potential of $L$, that is for any absolutely monotone potential $h$, the point $\widetilde{w}$ is a global minimum on $\mathbb{S}^{23}$ of $U_h(x):=\sum_{y\in L}h(x\cdot y)$, and all global minima are projections of points in $\Lambda(3)$ on $\mathbb{S}^{23}$ (see \cite{BDHSS2024}).

\subsection{On the orthogonal complement of the Golay code} Delsarte, Goethals and Seidel proved in \cite{DGS} (see Examples 4.7 and 9.3) that any spherical code with inner products in $[-1,\alpha] \cup [\beta,\gamma]$ (i.e., any
$(\alpha,\beta)$-avoiding $\gamma$-code) has cardinality at most
\begin{equation}\label{DGS_3bound} -\frac{n(1-\alpha)(1-\beta)(1-\gamma)}{\alpha+\beta+\gamma+n\alpha\beta\gamma} \end{equation}
provided that the numbers $\alpha$, $\beta$, and $\gamma$ satisfy 
\begin{equation}\label{DGSCond}\alpha+\beta+\gamma \leq 0, \ \ \alpha\beta+\beta\gamma+\gamma\alpha \geq -\frac{3}{n+2}, \ \ \alpha+\beta+\gamma+n\alpha\beta\gamma <0. \end{equation}
This result can be obtained via $f(t)=(t-\alpha)(t-\beta)(t-\gamma)$ in Theorem \ref{general-lp-max-codes}. We can write
\[ \mathcal{A}\left(n,\gamma;\left(\alpha,\beta\right)\right) \leq 
-\frac{n(1-\alpha)(1-\beta)(1-\gamma)}{\alpha+\beta+\gamma+n\alpha\beta\gamma} \]
under the above conditions. This bound is attained (Example 9.3 in \cite{DGS}) by a spherical code $C_G \subset \mathbb{S}^{22}$ which is obtained by standard mapping on $\mathbb{S}^{22}$ of the orthogonal complement of the binary Golay code. As expected, the code $C_G$ has three inner products, namely $\alpha=-9/23$, $\beta=-1/23$, and $\gamma=7/23$ (corresponding to the Hamming distances 16, 12, and 8, respectively). Therefore,
\[ \mathcal{A}\left(23,\frac{7}{23},\left(-\frac{9}{23},-\frac{1}{23}\right)\right)=2048 \]
from \cite{DGS}. 
The distance distribution is given by 
\[ F(C_{2048})=(253,1288,506). \]

\subsection{Codes on \texorpdfstring{$\mathbb{S}^{21}$}{S21}}\label{Codes_21} Next, we focus on codes in $\mathbb{S}^{21}$. We shall consider derived codes of the tight spherical $7$-design with $4600$ vectors $C_{1/2}  \subset \mathbb{S}^{22}$ considered in the previous section. If we fix a point of this code, there are four derived codes at inner products $1/3$, $0$, $-1/3$ and $-1$ with cardinality $891$, $2816$, $891$, and $1$, respectively. The first and the third derived codes are copies of the sharp code (a non-antipodal spherical 5-design and a maximal spherical $(1/4)$-code) with $891$ points, that is unique \cite{CK-NYJM}
and universally optimal \cite{CK2007}. So, our first code at this level will be the derived $C_{2816}:=C_0$ (i.e., equatorial) of $C_{1/2}$. It was shown in \cite{B1993} that $C_{2816}$ is a maximal antipodal spherical $(1/3)$-code and in \cite{BBDHSS2025} that it is universally optimal in the real projective space $\mathbb{RP}^{21}$.
For the purposes of this paper we need to observe that $C_{2816}$ is 
an antipodal spherical $5$-design on $\mathbb{S}^{21}$ of cardinality $|C_{2816}|=2816$ with inner products
\[ I(C_{2816})=\left\{-1, -\frac{1}{3},0, \frac{1}{3} \right\} \]
from \eqref{I-derived}. The distance distribution of this code is
\[ F(C_{2816})= (1,567,1680,567). \]

The other code we shall consider is derived when we fix a point $\widetilde{w}$, which is a global minimum of the discrete $h$ potential of sharp code $C_{1/2}$ above. There are four derived codes at inner products $\{ -\sqrt{5}/5, -\sqrt{5}/15, \sqrt{5}/15, \sqrt{5}/5\}$ with cardinality $275$, $2025$, $2025$, and $275$, respectively. From Theorem \ref{derived_codes_thm} we conclude they are $4$-designs. The $275$-code is a spectral embedding of the McLaughlin graph and is a sharp code (a tight $4$-design) and hence universally optimal.  
The other derived code is a spherical $4$-design $C_{2025} \subset \mathbb{S}^{21}$ of cardinality $|C_{2025}|=2025$ with
\[ I(C_{2025})=\left\{-\frac{4}{11},-\frac{1}{44}, \frac{7}{22} \right\}. \] 
It has distance distribution
\[ F(C_{2025})=(330,1232,462). \]

\subsection{Codes on \texorpdfstring{$\mathbb{S}^{15}$}{S15}}

We consider the kissing configuration of the Barnes--Wall lattice in 16 dimensions. It is an antipodal spherical $7$-design $C_{BW}$ of cardinality $4320$ with
\[ I(C_{BW})=\left\{-1,0, \pm \frac{1}{4}, \pm \frac{1}{2}\right\}. \]
The distance distribution is
\[ F(C_{BW})=(1,280,1024,1710,1024,280). \]

\subsection{Codes from strongly regular graphs} For the last portion of codes we need the notion of strongly regular graphs. A strongly regular graph (SRG) with parameters $(v, k, \lambda, \mu)$ is a graph with $v$ vertices such that every vertex has degree $k$,
every two adjacent vertices have exactly $\lambda$ common neighbors and every two non-adjacent vertices have exactly $\mu$ common neighbors. The adjacency matrix of an SRG has three eigenvalues, $k$, $e_1$, and $e_2$ ($e_1 \geq 0> e_2$) with multiplicities $1$, $n_1$ and $n_2$, respectively, where $n_1 + n_2 = v-1$. The 
eigenspaces that correspond to $e_1$ and $e_2$ contain embeddings of the graph which are spherical codes with usually good parameters. More precisely, every SRG has two embeddings into spheres $\mathbb{S}^{n_1-1}$ and $\mathbb{S}^{n_2-1}$ and these embeddings are $(n_1,v,2,2)$- and $(n_2,v,2,2)$-configurations, respectively. The two inner products are denoted by $p$ and $q$, as $p < 0 < q$. 

In order to apply the linear programming technique for these embeddings, we will require $p+q \leq 0$. Actually, at least one of these embeddings satisfies a stronger inequality.

\begin{lemma} \label{lm:SRGpq}
For every SRG the inequality $p+q < 0$ holds for at least one of the two embeddings. 
\end{lemma}

\begin{proof} The relations between the parameters of SRG and the inner products were described, for example, in Section 9.4 from \cite{EZ}. 
We have 
\[
p_1 = -\frac{1 + e_1}{v-k-1}, \quad q_1 = \frac{e_1}{k} 
\]
for the first embedding and
\[
p_2 = \frac{e_2}{k}, \quad q_2 = -\frac{1 + e_2}{v-k-1} 
\]
for the second one. It is easy to see that $p_1 + q_1 < 0$ if and only if $e_1 (v-2k-1) - k < 0$ and $p_2 + q_2 < 0$ if and only if $e_2 (v-2k-1) - k < 0$. Since the eigenvalues satisfy $e_1 \geq 0> e_2$, it follows that (whatever is the sign of $(v-2k-1)$) at least one of the embeddings has inner products $p$ and $q$ such that $p+q < 0$.
\end{proof}
We note that sometimes the condition $p+q \leq 0$ is satisfied by both embeddings. For example, this is the case with the Schl{\"a}fli SRG with parameters $(27, 16, 10, 8)$ and the so-called conference graphs (where $k = (v-1)/2$).

Recently Cohn, de Laat and Leijenhorst~\cite{CLL} used semidefinite programming to show that certain embeddings of SRG have maximal cardinality among the $q$-codes. They conjectured that such an approach may be used for many triangle-free SRG.

\begin{conjecture}[Cohn--de Laat--Leijenhorst,~\cite{CLL}]
Let $G$ be a connected triangle-free SRG other
than a complete bipartite graph, and let $C$ be the code obtained by the embedding of $G$ into its eigenspace with the smallest eigenvalue $e_2$. Then three-point (SDP) bounds prove that $C$ is a maximal spherical code.
\end{conjecture}

In Section 7 we prove universal optimality of $[-1,p)$-avoiding SRG codes under the assumption $p + q \leq 0$.

\section{Maximal \texorpdfstring{$T$}{T}-avoiding codes}

\subsection{Codes on \texorpdfstring{$\mathbb{S}^{22}$}{S22}}

We start our consideration of maximal codes with $T$-avoiding codes on $\mathbb{S}^{22}$ with maximal inner product $s=7/15$. We shall prove that the code $C_{1/4}=(23,47104,7/15)$ is a maximal $T$-avoiding $(7/15)$-code for
each
\begin{equation} \label{T-codes}
T\in \left\{\left(-\frac{3}{5},-\frac{1}{3}\right), \left(-\frac{1}{3}, -\frac{1}{15}\right), \left(-\frac{1}{15}, \frac{1}{5}\right)\right\}. \end{equation} 

\begin{thm} \label{47104-codes-t1}
We have $\mathcal{A}(23,7/15;T)=47104$ for each $T$ from \eqref{T-codes}.
Consequently, $C_{1/4}$ is a maximal $T$-avoiding spherical $(7/15)$-code
for these $T$. 
\end{thm} 

\begin{proof}
We explain in detail the case $T=(-1/3,-1/15)$. 
We apply Theorem \ref{general-lp-max-codes} with the polynomial
\[ f(t)= \left(t + \frac{3}{5}\right)^2 \left(t+\frac{1}{3}\right) \left(t+\frac{1}{15}\right) \left(t-\frac{1}{5}\right)^2 \left(t-\frac{7}{15}\right). \]

Since $f(t) \leq 0$ for $t \in [-1,7/15] \setminus T$, the condition (A1) is satisfied. 
We verify (A2) by presenting explicitly the Gegenbauer expansion of $f(t)$. Since
\[ f(t)=\sum_{i=0}^{7} f_iP_i^{(23)}(t), \]
where 
\[ f_0=\frac{256}{9703125}, \ f_1=\frac{3328}{7340625}, \ f_2=\frac{2962432}{506503125}, \ f_3=\frac{13274624}{379265625}, \]
\[ f_4=\frac{450208}{4708125}, \ f_5=\frac{13408}{58725}, \ f_6=\frac{50336}{121365}, \ f_7=\frac{416}{899}, \] 
the condition (A2) is satisfied as well. 

It remains to see that $f(1)=524288/421875$ and, therefore,
\[ 47104 \leq \mathcal{A}\left(23,\frac{7}{15};\left(-\frac{1}{3}, -\frac{1}{15}\right)\right) \leq \frac{f(1)}{f_0}=47104.\]
This completes the proof in this case. The technical details for the other two cases, $\mathcal{A}\left(23,7/15; \left(-1/15,1/5\right)\right)=
\mathcal{A}\left(23,7/15; \left(-3/5,-1/3\right)\right)=47104$,
are given in the Appendix.  
\end{proof}

The details of the next two proofs (polynomials and their Gegenbauer expansions) can be found in the Appendix. We will first see that the code $C_0=(23,93150,1/2)$ is a maximal $T$-avoiding $(1/2)$-code for $T$ being union of two open intervals. 

\begin{thm} \label{93150-codes-t1}
We have $\mathcal{A}\left(23,1/2;\left(-1/2,-1/4\right) \cup \left(0,1/4\right)\right)=93150$.
Consequently, $C_{0}$ is a maximal $\left(-1/2,-1/4\right) \cup \left(0,1/4\right)$-avoiding spherical $(1/2)$-code. 
\end{thm} 

Next we have two optimality results for the code $C_{\sqrt{6}/6}$. 

\begin{thm} \label{11178-codes-t1}
We have $\mathcal{A}\left(23,2/5;\left(-1/2,-1/5\right)\right)=
\mathcal{A}\left(23,2/5; \left(-1/5,1/10\right)\right)=11178$.
Consequently, $C_{\sqrt{6}/6}$ is a maximal $\left(-1/2,-1/5\right)$-avoiding spherical $(2/5)$-code and
a maximal $\left(-1/5,1/10\right)$-avoiding spherical $(2/5)$-code.
\end{thm} 

We next prove that the code $C_{\sqrt{6}/12}$ is a maximal $(-13/23,-7/23) \cup (5/23,6/23)$-avoiding $(11/23)$-code. 
The proof contains an additional trick based on the fact that $M_7(C_{\sqrt{6}/12})=0$ while $C_{\sqrt{6}/12}$ itself is a non-antipodal spherical $5$-design. 

\begin{thm} \label{48600-codes-t}
We have $\mathcal{A}\left(23,11/23;\left(-13/23,-7/23\right) \cup \left(5/23,6/23\right)\right)=48600$. Consequently, $C_{\sqrt{6}/12}$ is a maximal $\left(-13/23,-7/23\right) \cup \left(5/23,6/23\right)$-avoiding spherical $(11/23)$-code. 
\end{thm}

\begin{proof}
Consider the 7th degree polynomial
\[ f(t)= \left(t + \frac{13}{23}\right) \left(t + \frac{7}{23}\right)
\left(t+\frac{1}{23}\right)^2 \left(t-\frac{5}{23}\right) 
\left(t-\frac{6}{23}\right) \left(t-\frac{11}{23}\right).  \]
It satisfies (A1) with the above $T$. It also satisfies (A2) since its Gegenbauer expansion $f(t)=\sum_{i=0}^{7} f_iP_i^{(23)}(t)$ has
\[ f_0=\frac{235008}{17024127235}, \ f_1=\frac{125781728}{965934175725}, \ 
f_2=\frac{3768160}{1332323001}, \ f_3=\frac{85802112}{6289426475}, \] 
\[ f_4=\frac{10208}{547515}, f_5=\frac{23872}{138069}, \ f_6=0, \ 
f_7=\frac{416}{899}. \]
Finally, $f(1)=2284277760/3404825447$ and, consequently,  
\[ 48600 \leq \mathcal{A}\left(23,\frac{11}{23};
\left(-\frac{13}{23},-\frac{7}{23}\right) \cup \left(\frac{5}{23},\frac{6}{23}\right)\right) \leq \frac{f(1)}{f_0}=48600. \] 

This bound is attained by the code $C_{\sqrt{6}/12}$. This code is a non-antipodal $5$-design but has $M_7(C_{\sqrt{6}/12})=0$ (here $f_6=0$ comes into play). The extra (not in $I(C_{\sqrt{6}/12})$) zero $6/23$ not only provides $f_6=0$ but also somehow reduces the forbidden interval.   
\end{proof}

\subsection{Codes on \texorpdfstring{$\mathbb{S}^{21}$}{S21}}

On $\mathbb{S}^{21}$ we consider the codes $C_{2816}$ and $C_{2025}$. 

\begin{thm} \label{2816-codes-t1}
We have $\mathcal{A}\left(22,1/3;\left(-1,-1/3\right)\right)=
\mathcal{A}\left(22,1/3; \left(-1/3,0\right)\right)=2816$.
Consequently, $C_{2816}$ is a maximal $\left(-1,-1/3\right)$-avoiding spherical $(1/3)$-code and a maximal $(-1/3,0)$-avoiding spherical $(1/3)$-code. 
In particular, $C_{2816}$ is a maximal antipodal $(1/3)$-code (Theorem 6 in \cite{B1993}). 
\end{thm} 

The next result is an analog of Example 9.3 from the paper of Delsarte, Goethals, and Seidel \cite{DGS} about 3-distance codes. 

\begin{thm} \label{2025-codes-t2}
We have $\mathcal{A}\left(22,7/22;\left(-4/11,-1/44\right)\right)=2025$. 
Consequently, $C_{2025}$ is a maximal $(-4/11,-1/44)$-avoiding spherical $(7/22)$-code. 
\end{thm} 

\subsection{Codes on \texorpdfstring{$\mathbb{S}^{15}$}{S15}}

Here we consider the code $C_{BW}$.   

\begin{thm} \label{4320-codes-t4}
We have $\mathcal{A}\left(16,1/2;\left(-1/2,-1/4\right) \cup \left(0,1/4\right)\right)=4320$.
Consequently, $C_{BW}$ is a maximal $(-1/2,-1/4) \cup (0,1/4)$-avoiding spherical $(1/2)$-code. 
\end{thm}

\section{Minimum \texorpdfstring{$T$}{T}-avoiding designs (tight \texorpdfstring{$T$}{T}-avoiding designs)}

We first consider in detail the derived code $C_{1/4} \subset \mathbb{S}^{22}$. Let $T_1:=(-3/5,-1/3) \cup (1/5,7/15)$. We prove that $C_{1/4}$ is a minimal $T_1$-avoiding spherical $7$-design on $\mathbb{S}^{22}$. 

\begin{thm} \label{47104-des-t1}
We have $\mathcal{B}(23,6;T_1) = \mathcal{B}(23,7;T_1)=47104$. 
Consequently, the code $C_{1/4}$ is a tight $T_1$-avoiding spherical $7$-design. 
\end{thm} 

\begin{proof}
We apply Theorem \ref{general-lp-des} with the polynomial
\[ f(t)= \left(t + \frac{3}{5}\right) \left(t+\frac{1}{3}\right) \left(t+\frac{1}{15}\right)^2 \left(t-\frac{1}{5}\right) \left(t-\frac{7}{15}\right). \]
Since $f(t) \geq 0$ for $t \in [-1,1] \setminus T_1$ and $\deg{f} = 6$, the conditions (B1) and (B2) are satisfied. 
So the inequality $\mathcal{B}(23,6;T_1) \geq 47104$ follows from $f(1)=262144/253125$ and $f_0=128/5821875$, whence $f(1)/f_0=47104=|C_{1/4}|$. 
By definition $\mathcal{B}(23,7;T_1) \geq \mathcal{B}(23,6;T_1)$ and the example of $C_{1/4}$ gives $47104 \geq \mathcal{B}(23,7;T_1)$.
\end{proof}

The details for the next theorems in this section are presented in the Appendix.

\begin{thm} \label{47104-des-t2}
Let $T_2:=(-9/15,-5/15) \cup (-1/15,3/15)$.
Then $\mathcal{B}(23,6;T_2) = \mathcal{B}(23,7;T_2) = 47104$.
Consequently, the code $C_{1/4}$ is a tight $T_2$-avoiding spherical $7$-design. 
\end{thm} 

We next consider the derived code $C_0 \subset \mathbb{S}^{22}$. Let 
\begin{equation}\label{T_93150des}
T \in \left\{\left(-\frac{1}{2},-\frac{1}{4}\right) \cup 
\left(\frac{1}{4},\frac{1}{2}\right), \left(-\frac{1}{4},0\right) \cup \left(\frac{1}{4},\frac{1}{2}\right), \left(-\frac{1}{2},-\frac{1}{4}\right) \cup \left(0,\frac{1}{4}\right)\right\}.
\end{equation}
This choice of $T$ will be used several times in this and the next section. 

\begin{thm} \label{93150-des-t1}
Let $T$ be any of the sets from \eqref{T_93150des}. Then $\mathcal{B}(23,7;T)=93150$. 
Consequently, the code $C_0$ is a tight $T$-avoiding spherical $7$-design
for each $T$ from \eqref{T_93150des}. 
\end{thm} 

Let $T:=(-1/2,-1/5) \cup (1/10,2/5)$. We prove that $C_{\sqrt{6}/6}$ is a minimal $T$-avoiding spherical $5$-design on $\mathbb{S}^{22}$. 

\begin{thm} \label{11178-des-t1}
We have $\mathcal{B}(23,4;T) = \mathcal{B}(23,5;T) = 11178$. 
Consequently, the code $C_{\sqrt{6}/6}$ is a tight $T$-avoiding spherical $5$-design. 
\end{thm}

We proceed with the codes $C_{2816}$ and $C_{2025}$ on $\mathbb{S}^{21}$. 

\begin{thm} \label{2816-des-t1}
We have $\mathcal{B}\left(22,5;\left(0,1/3\right)\right)=
\mathcal{B}\left(22,5;\left(-1/3, 0\right)\right)=2816$.
Consequently, the code $C_{2816}$ is a tight $(0,1/3)$-avoiding spherical $5$-design and a tight $(-1/3,0)$-avoiding spherical $5$-design. 
\end{thm} 

\begin{thm} \label{2025-des-t1}
We have $\mathcal{B}(22,4;(-1/44,7/22))=\mathcal{B}(22,4;(-4/11,-1/44))=
2025$. Consequently, the code $C_{2025}$ is a tight $(-1/44,7/22)$-avoiding spherical $4$-design and a tight $(-4/11,-1/44)$-avoiding spherical $4$-design. 
\end{thm} 

We conclude this section with three results for the code $C_{BW}$ from the Barnes--Wall lattice in $\mathbb{R}^{16}$. 

\begin{thm} \label{4320-des-t1} Let $T$ be any of the sets from \eqref{T_93150des}. 
Then $\mathcal{B}(16,7;T)=4320$. Consequently, the code $C_{BW}$ is a tight $T$-avoiding spherical $7$-design for each $T$ from \eqref{T_93150des}. 
\end{thm}

\section{Universally optimal \texorpdfstring{$T$}{T}-avoiding codes}

\subsection{Energy bounds for derived codes on \texorpdfstring{$\mathbb{S}^{22}$}{S22} determined by a point in \texorpdfstring{$\Lambda(2)$}{L2}}
Fixing a point in $\Lambda(2)$ we obtain three derived codes, the sharp code $C_{1/2}$ with $4600$ points, $C_{1/4}$ with $47104$ points, and $C_0$ with $93150$ points. 

For the analysis of energy bounds on $\mathbb{S}^{22}$ with $47104$ points we shall consider $T$-avoiding codes where $T$ is any of the four subintervals determined by the five points of $I(C_{1/4})$, namely \begin{equation}\label{T_47104}
T\in \left\{\left(-\frac{3}{5},-\frac{1}{3}\right), \left(-\frac{1}{3},-\frac{1}{15}\right), \left(-\frac{1}{15}, \frac{1}{5}\right), \left(\frac{1}{5},\frac{7}{15}\right)\right\}.
\end{equation}

\begin{thm} \label{47104-en-t1} Let $h$ be absolutely monotone with $h^{(8)}>0$ in $(-1,1)$ and $C \subset \mathbb{S}^{22}$ be a $T$-avoiding spherical code with $|C|=47104$, where $T$ is given by \eqref{T_47104}. Then 
\begin{equation}\label{E47104} 
\begin{split}
E_h(C) \geq \mathcal{E}_h(23,47104;T) &= 47104\left( 275h\left(-\frac{3}{5}\right)+7128h\left(-\frac{1}{3}\right) + 22275h\left(-\frac{1}{15}\right) \right.\\ 
& \ \ \ \quad \quad \quad \left. +15400h\left(\frac{1}{5}\right) +2025h\left(\frac{7}{15}\right)\right).
\end{split}
\end{equation}
The equality is attained when $C$ is a spherical 7-design with inner-product set $I(C_{1/4})$ (and hence a distance-invariant code) and distance distribution $F(C_{1/4})$. In particular, the second derived code in the minimal vectors of the Leech lattice is $T$-avoiding universally optimal on $\mathbb{S}^{22}$.
\end{thm} 
\begin{remark}\label{AbsM}
    As the proof shows, the $T$-avoiding universal optimality holds for a larger class, namely potentials that have $h^{(j)}\geq 0$, $j=1,2,\dots,7$, and $h^{(8)}>0$. Similar observations hold for all theorems in this section.
\end{remark}
\begin{proof}
We shall apply Theorem \ref{general-lp-energy} with an appropriately chosen polynomial $H_7(t;h,T)$, which shall interpolate the potential function $h$ at the endpoints of $T$ and the potential function $h$ and its derivative $h'$ at the other points of $I(C_{1/4})$. Let $\mathcal{I}:=\{t_1,\dots,t_8\}$ denote the multi-set associated with the Hermite/Lagrange interpolation in increasing order. 

We will illustrate the argument when $T=(-3/5,-1/3)$ and refer to the Appendix for the other three choices of $T$. For this choice the multi-set is 
$$\mathcal{I}=\left\{-\frac{3}{5},-\frac{1}{3},-\frac{1}{15},-\frac{1}{15},\frac{1}{5}, \frac{1}{5},\frac{7}{15},\frac{7}{15}\right\}.$$
Newton's interpolation formula implies that
\[
    H_7(t;h,\mathcal{I})= h(t_1)+\sum_{i=1}^7 h[t_1,\dots,t_{i+1}](t-t_1)\dots(t-t_i),
\]
where $h[t_1,\dots,t_{i+1}]$,  $i=1,\dots,7$, are the corresponding divided differences and 
$$P_i(t):=(t-t_1)\dots(t-t_i), \quad i=1,\dots,7$$ are the associated partial products. The absolute monotonicity of $h$ implies that all divided differences are non-negative. 

We shall verify directly that all partial products $P_i(t)$, $i=1,\dots,7$, have positive Gegenbauer coefficients, from which we shall derive that $H_7 (t;h,\mathcal{I})$ has positive Gegenbauer coefficients. Actually, as any linear factors $t+c$, $c\geq 0$, are positive definite and product of positive definite polynomials is positive definite (see \cite{Gas70}; this is the so-called Krein condition in \cite{Lev98}), we only need to check partial products containing factors of the type $t-c$ with $c>0$, which in this case are $P_5$, $P_6$, and $P_7$. We compute, respectively, that
\[ P_5(t)=\left(t + \frac{3}{5}\right) \left(t+\frac{1}{3}\right) \left(t+\frac{1}{15}\right)^2 \left(t-\frac{1}{5}\right)=\sum_{i=0}^5 f_iP_i^{(23)}(t), \]
with
\[f_0=\frac{1096}{388125}, \ f_1=\frac{104}{3375}, \ 
f_2=\frac{11704}{77625}, \ f_3=\frac{66088}{163125}, \ f_4=\frac{2288}{3375}, \ f_5=\frac{176}{261} ;\]
\[ P_6(t)=\left(t + \frac{3}{5}\right) \left(t+\frac{1}{3}\right) \left(t+\frac{1}{15}\right)^2 \left(t-\frac{1}{5}\right)^2=\sum_{i=0}^6 f_iP_i^{(23)}(t), \]
where
\[f_0=\frac{1504}{1940625}, \ f_1=\frac{736}{84375}, \ f_2=\frac{1497056}{33766875}, \ f_3=\frac{369952}{2446875}, \ 
f_4=\frac{523072}{1569375},\] 
\[f_5=\frac{352}{783}, \ f_6=\frac{4576}{8091} ;\]
\[ P_7(t)=\left(t + \frac{3}{5}\right) \left(t+\frac{1}{3}\right) \left(t+\frac{1}{15}\right)^2 \left(t-\frac{1}{5}\right)^2\left( t-\frac{7}{15}\right)=\sum_{i=0}^7 f_iP_i^{(23)}(t), \]
for which the Gegenbauer coefficients are: 
\[f_0=\frac{512}{29109375}, \ f_1=\frac{1024}{4078125}, \ f_2=\frac{250624}{56278125}, \ f_3=\frac{18436352}{1137796875}, f_4=\frac{80608}{1569375}, \]
\[f_5=\frac{10592}{58725}, \ f_6=\frac{4576}{40455}, \ f_7= \frac{416}{899}. \]
The verifications of the positive-definiteness of the partial products for the other three choices of $T$ are found in the appendix. Note that $P_7$ is exactly the polynomial for Theorem \ref{47104-codes-t1} in the case $T=(-3/5,-1/3)$. 

To complete the argument we utilize the Hermite error formula, which yields (recall that $T$ is an interval)
\[h(t)-H_7 (t;h,\mathcal{I})=\frac{h^{(8)}(\xi)}{8!}(t-t_1)\dots(t-t_8)\geq 0 \quad \quad {\rm for}\quad \quad t\in [-1,1)\setminus T.\]
For example, when $T=(-3/5,-1/3)$ the Hermite error formula is
\[h(t)-H_7 (t;h,\mathcal{I})=\frac{h^{(8)}(\xi)}{8!} \left (t + \frac{3}{5} \right)\left (t + \frac{1}{3} \right)\left (t + \frac{1}{15} \right)^2 \left (t - \frac{1}{5} \right)^2 \left (t - \frac{7}{15} \right)^2 ,\]
which is clearly nonnegative on $[-1,1)\setminus(-3/5,-1/3)$.

Therefore, $H_7(t;h,\mathcal{I})$ satisfies {\rm (E1)} and {\rm (E2)} in Theorem \ref{general-lp-energy} and we have
\[ E_h (C)\geq 47104^2 \left( (H_7)_0-\frac{H_7(1)}{47104}\right),\]
which, after application of \eqref{QF}, implies the bound \eqref{E47104}. 

If equality holds for some code $C$, then all moments $M_i(C)$, $i=1,\dots,7$ vanish, so the code has to be a $7$-design. Moreover, the code has to be with inner products in $I(C_{1/4})$ (and thus distance-invariant), and has to have the same frequency distribution (see \eqref{QF}), which concludes the proof.
\end{proof}

We also derive that $C_0$ is a universally optimal $T$-avoiding code, where $T$ is as in \eqref{T_93150des}. 

\begin{thm} \label{93150-en-t1}
Let $h$ be absolutely monotone with $h^{(8)}>0$ in $(-1,1)$. 
Let $C \subset \mathbb{S}^{22}$ be a $T$-avoiding spherical code with $|C|=93150$, where $T$ is given by \eqref{T_93150des}. Then 
\begin{equation*} 
\begin{split}
E_h(C) \geq \mathcal{E}_h(23,93150;T)&= 93150\left(2464\left( h\left(-\frac{1}{2}\right) +h\left (\frac{1}{2}\right)\right)\right. \\ 
& \quad \left. +22528\left( h\left(-\frac{1}{4}\right) +h\left (\frac{1}{4}\right)\right)+43164 h\left(0\right) +h(-1)\right).
\end{split}
\end{equation*}
The equality is attained when $C$ is a spherical 7-design with inner-product set $I(C_{0})$ (and hence a distance-invariant code) and distance distribution $F(C_0)$. In particular, the third derived code in the minimal vectors of the Leech lattice is $T$-avoiding universally optimal on $\mathbb{S}^{22}$.
\end{thm} 

\begin{proof}
The proof proceeds as in Theorem \ref{47104-en-t1}. There are three different interpolating multi-sets 
\[
\mathcal{I}=\{t_1, t_2, t_3, t_4, t_5, t_6, t_7, t_8\},
\]
for each of the cases in \eqref{T_93150des}.
In the appendix, we verify that the corresponding partial products $P_i$, $i=1,\dots,7$, and therefore $H_7 (t;h,\mathcal{I})$, have positive Gegenbauer coefficients. The Hermite interpolation error reveals that inequality $H_7 (t;h,\mathcal{I})\leq h(t)$ holds on $[-1,1)\setminus T$ and the proof is concluded similarly to Theorem \ref{47104-en-t1}.
\end{proof}

\subsection{Energy bounds for derived codes on \texorpdfstring{$\mathbb{S}^{22}$}{S22} determined by a point in \texorpdfstring{$\Lambda(3)$}{L3}} Recall that in this case the fixed vector $\widetilde{w}$ is the projection on a fixed point $W\in \Lambda(3)$ and we obtain three derived codes, the sharp code $C_{\sqrt{6}/4}$ with $552$ points, and $C_{\sqrt{6}/6}$ and $C_{\sqrt{6}/12}$ with $11178$ and $48600$ points, respectively. Note that by Theorem \ref{derived_codes_thm} all these codes are $5$-designs.

For the analysis of energy bounds on $\mathbb{S}^{22}$ with $11178$ points we shall consider $T$-avoiding codes where $T$ is any of the three subintervals determined by the four points of $I(C_{\sqrt6/6})$, namely 
\begin{equation}\label{T_11178}
T \in \left \{ \left(-\frac{1}{2},-\frac{1}{5}\right), \left(-\frac{1}{5},\frac{1}{10}\right), \left(\frac{1}{10}, \frac{2}{5}\right) \right\}.
\end{equation}
The proof of the following theorem follows closely the approach of Theorem \ref{47104-en-t1} by selecting appropriate interpolation polynomial of degree $5$. The positive-definiteness of the relevant partial products is verified in the appendix.
\begin{thm} \label{11178-en-t1} Let $h$ be absolutely monotone with $h^{(6)}>0$ in $(-1,1)$ and $C \subset \mathbb{S}^{22}$ be a $T$-avoiding spherical code with $|C|=11178$, where $T$ is given by \eqref{T_11178}. Then 
\begin{equation*} 
\begin{split}
E_h(C) \geq \mathcal{E}_h(23,11178;T) &= 11178\left( 352h\left(-\frac{1}{2}\right)+4125h\left(-\frac{1}{5}\right)  \right.\\ 
& \ \ \ \  \quad \quad \quad \left. +5600h\left(\frac{1}{10}\right) +1100h\left(\frac{2}{5}\right)\right).
\end{split}
\end{equation*}
The equality is attained when $C$ is a spherical $5$-design with inner-product set $I(C_{\sqrt{6}/6})$ (and hence a distance-invariant code) and distance distribution $F(C_{\sqrt{6}/6})$. In particular, the second derived code with respect to $\widetilde{w}$ of the minimal vectors of the Leech lattice is $T$-avoiding universally optimal code on $\mathbb{S}^{22}$.
\end{thm} 

For the third derived code $C_{\sqrt{6}/12}$ with cardinality $48600$ we shall consider the following avoiding sets
\[
T\in \left\{ \left(-\frac{13}{23},-\frac{7}{23}\right)\cup\left(-\frac{1}{23},\frac{5}{23}\right),\left(-\frac{13}{23},-\frac{7}{23}\right)\cup\left(\frac{5}{23},\frac{11}{23}\right),\left(-\frac{7}{23},-\frac{1}{23}\right)\cup\left(\frac{5}{23},\frac{11}{23}\right) \right\}.
\]
As was the case with Theorem \ref{11178-en-t1} we shall derive the positive definiteness of the partial products in the appendix, the remainder of the proof following closely Theorem \ref{47104-en-t1}.
\begin{thm} \label{48600-en-t1} Let $h$ be absolutely monotone with $h^{(6)}>0$ in $(-1,1)$ and $C \subset \mathbb{S}^{22}$ be a $T$-avoiding spherical code with $|C|=48600$, where $T$ is given above. Then 
\begin{equation*} 
\begin{split}
E_h(C) \geq \mathcal{E}_h(23,48600;T) &= 48600\left( 506h\left(-\frac{13}{23}\right)+8855h\left(-\frac{7}{23}\right) \right.\\ 
& \  \quad \left. +23046h\left(-\frac{1}{23}\right) +14421h\left(\frac{5}{23}\right) +1771h\left(\frac{11}{23}\right)\right).
\end{split}
\end{equation*}
The equality is attained when $C$ is a spherical $5$-design with inner-product set $I(C_{\sqrt{6}/12})$ (hence a distance-invariant code) and distance distribution $F(C_{\sqrt{6}/12})$. In particular, the third derived code with respect to $\widetilde{w}$ of the minimal vectors of the Leech lattice is $T$-avoiding universally optimal code on $\mathbb{S}^{22}$.
\end{thm} 

\subsection{Energy bounds for the orthogonal complement of the Golay code}
In their seminal paper from 1977 Delsarte, Goethals, and Seidel  solved the maximal cardinality problem for $(-9/23,-1/23)$-avoiding codes and found that the orthogonal complement of the Golay code $C_G$ of cardinality $2048$ is optimal in this class (see \cite[Example 9.3]{DGS}). In this subsection we shall derive the universal optimality of $C_G$ among all $T$-avoiding codes, where 
\begin{equation} \label{T_2048}
  T \in \left \{\left(-\frac{9}{23},-\frac{1}{23}\right), \left (-\frac{1}{23},\frac{7}{23}\right) \right\}.
\end{equation}

First, we shall prove the following general theorem for codes with three inner products, or $\{\alpha, \beta, \gamma\}$-codes, where without loss of generality we assume $-1\leq \alpha\leq \beta \leq \gamma <1$. Note that Theorem \ref{general-lp-energy} (E2) has $f_i\geq 0$, $i\geq 1$, which is why we are not concerned with the $0$-th coefficient of the interpolation polynomial and hence of any of the partial products. This allows us to relax Delsarte--Goethals--Seidel conditions \eqref{DGSCond} and exclude the strict inequality from the assumptions on $\alpha, \beta, \gamma$.

\begin{thm}\label{3_energy}
Let $h$ be an absolutely monotone potential with $h^{(4)}>0$ and let $-1\leq \alpha\leq \beta \leq \gamma <1$ be such that there exists an $\{\alpha,\beta,\gamma\}$-code $B$ that is a spherical $3$-design on $\mathbb{S}^{n-1}$. Suppose
\begin{equation} \label{CondB}
\alpha\beta+\beta\gamma+\gamma\alpha \geq -\frac{3}{n+2}, \ \ \alpha+\beta+\gamma\leq 0. 
\end{equation}
   Then for any code $C$ with cardinality $|C|=|B|$ that is $(\alpha,\beta)$-avoiding or $(\beta,\gamma)$-avoiding the following energy bound holds:
\[
E_h(C) \geq \mathcal{E}_h(n,|B|;T) = |B|\left( m_\alpha h(\alpha)+m_\beta h(\beta)+m_\gamma h(\gamma) \right),
\]
where $T=(\alpha,\beta)$ or $T=(\beta,\gamma)$, respectively, and $F(B)=(m_\alpha,m_\beta, m_\gamma)$ is the distance distribution of the (distance-invariant) code $B$. If equality holds, then $C$ is a spherical $3$-design with inner-product set 
$I(C)=\{\alpha,\beta,\gamma\}$ and distance distribution $F(B)$.
\end{thm}

\begin{proof} Note that distance invariance of $B$ follows from Theorem \ref{thm:DGSdistances}. 

For $(\alpha,\beta)$-avoiding codes we select the interpolation multi-set
$\mathcal{I}=\{\alpha,\beta,\gamma,\gamma\}$. The partial products whose positive definiteness we seek are $t-\alpha$, $(t-\alpha)(t-\beta)$, and $(t-\alpha)(t-\beta)(t-\gamma)$. We may assume $\gamma>0$ (otherwise $B$ cannot be a 3-distance spherical 3-design).
We remark that the conditions in \eqref{CondB} yield the non-negativity of $f_1$ and $f_2$ (it is clear that $f_3\geq 0$) in the expansion of the polynomial
\[ (t-\alpha)(t-\beta)(t-\gamma)=f_0+f_1 P_1^{(n)} (t)+f_2 P_2^{(n)} (t)+f_3 P_3^{(n)} (t). \]
We also have from $\gamma>0$ that $\alpha+\beta<0$ and therefore the partial product
\[(t-\alpha)(t-\beta)=t^2-(\alpha+\beta)t+\alpha\beta=\frac{n-1}{n}P_2^{(n)}(t)-(\alpha+\beta)P_1^{(n)}(t)+\alpha\beta+\frac{1}{n}, \]
has positive $f_1$ and $f_2$.
Finally, as $\gamma>0$, then $\alpha<0$ because of  $\alpha+\beta+\gamma\leq 0$ and $(t-\alpha)$ has positive Gegenbauer coefficients as well. 

For $(\beta,\gamma)$-avoiding codes we select the interpolation multi-set
$$\mathcal{I}=\{\alpha,\alpha,\beta,\gamma\},$$
and the partial products are $(t-\alpha)$, $(t-\alpha)^2$, and $(t-\alpha)^2(t-\beta)$. The first two are clearly positive definite. The last  we write as
\[ (t-\alpha)^2(t-\beta)=(t-\alpha)(t-\beta)(t-\gamma)+(\gamma-\alpha)(t-\alpha)(t-\beta),\]
so it has non-negative coefficients $f_1, f_2, f_3$ (note the sign of $f_0$ does not play a role in the energy LP theorems).
The remainder of the proof follows closely Theorem \ref{47104-en-t1}.
\end{proof}

As an application we obtain the following theorem. 
\begin{thm} \label{2048-en-t1}
Let $h$ be absolutely monotone with $h^{(4)}>0$. Then for any $T$-avoiding code on $\mathbb{S}^{22}$ with $2048$ points, with $T$ given in \eqref{T_2048} we have
\[ E_h (C)\geq\mathcal{E}_h(23,2048;T)= 2048\left( 253h \left(-\frac{9}{23} \right) + 1288h\left(-\frac{1}{23}\right)+ 506h \left(\frac{7}{23}\right) \right).\]
Consequently, the code $C_{G}$ is universally optimal in the class of such $T$-avoiding spherical codes on $\mathbb{S}^{22}$. If equality is attained for any other code it has to be a spherical $3$-design with the same inner products and distance distribution as $C_{G}$.
\end{thm} 

\begin{proof}
The conditions in \eqref{CondB} are verified directly (and were established in \cite[Example 9.3]{DGS}).   
\end{proof}

\subsection{Energy bounds for derived codes on \texorpdfstring{$\mathbb{S}^{21}$}{S21}} 
We now consider the two derived codes from Subsection \ref{Codes_21} $C_{2816}$ and $C_{2025}$. Recall that the first is an antipodal spherical $5$-design and the second is a spherical $4$-design.

The collection of $T$-avoiding intervals for $C_{2816}$ is
\begin{equation}\label{T_2816}
    T\in \left\{ \left(-1,-\frac{1}{3}\right), \left(-\frac{1}{3},0\right), \left(0,\frac{1}{3}\right) \right\}.
\end{equation}
The positive definiteness of the interpolation polynomials associated with each case is shown in the appendix, with the remainder of the proof similar to the proof of Theorem \ref{47104-en-t1}.
\begin{thm} \label{2816-en-t1}
Let $h$ be absolutely monotone with $h^{(6)}>0$. Then for any $T$-avoiding code on $\mathbb{S}^{21}$ with $2816$ points, with $T$ given in \eqref{T_2816} we have
\[ E_h (C)\geq\mathcal{E}_h(22,2816;T)= 2816\left( h(-1)+567h \left(-\frac{1}{3} \right) + 1680h\left(0\right)+ 567h \left(\frac{1}{3}\right) \right).\]
Consequently, the code $C_{2816}$ is universally optimal in the class of such $T$-avoiding spherical codes on $\mathbb{S}^{21}$. If equality is attained for any other code, it must be a spherical $5$-design with the same inner products and distance distribution as $C_{2816}$.
\end{thm} 

The code $C_{2025}$ is another example of an $\{\alpha,\beta,\gamma\}$-code and the following theorem is a corollary of Theorem \ref{3_energy}. Moreover, it is another configuration attaining the bound \eqref{DGS_3bound}.

\begin{thm} \label{2025-en-t1}
Let $h$ be absolutely monotone with $h^{(4)}>0$. Then for any $T$-avoiding code $C$ on $\mathbb{S}^{21}$, $T\in \{ (-4/11,-1/44), (-1/44,7/22)\}$, with $|C|=2025$ we have
\[ E_h (C)\geq\mathcal{E}_h(22,2025;T)= 2025\left( 330h \left(-\frac{4}{11} \right) + 1232h\left(-\frac{1}{44}\right)+ 462h \left(\frac{7}{22}\right) \right).\]
Consequently, the code $C_{2025}$ is universally optimal in the class of such $T$-avoiding spherical codes on $\mathbb{S}^{21}$. If equality is attained for any other code $C$, it has to be a spherical $3$-design with the same inner products and distance distribution as $C_{2025}$.
\end{thm} 

\subsection{Energy bounds for derived codes on \texorpdfstring{$\mathbb{S}^{15}$}{S15}} 

We prove that the code $C_{BW}$ is a universally optimal $T$-avoiding code, where $T$ is as in \eqref{T_93150des}.

\begin{thm} \label{4320-en-t1}
Let $h$ be absolutely monotone with $h^{(8)}>0$ in $(-1,1)$. 
Let $C \subset \mathbb{S}^{15}$ be a $T$-avoiding spherical code with $|C|=4320$, where $T$ is given by \eqref{T_93150des}. Then 
\begin{equation*} 
\begin{split}
E_h(C) \geq \mathcal{E}_h(16,4320;T)&= 4320\left(280\left( h\left(-\frac{1}{2}\right) +h\left (\frac{1}{2}\right)\right)\right. \\ 
& \quad \left. +1024\left( h\left(-\frac{1}{4}\right) +h\left (\frac{1}{4}\right)\right)+1710 h\left(0\right) +h(-1)\right).
\end{split}
\end{equation*}
The equality is attained when $C$ is a spherical 7-design with set of inner products $I(C_{BW})$ (and hence a distance-invariant code) and with distance distribution $F(C_{BW})$. In particular, the kissing configuration of the Barnes--Wall lattice is $T$-avoiding (with $T$ as in \eqref{T_93150des}) universally optimal on $\mathbb{S}^{15}$.
\end{thm} 

\subsection{Energy bounds for spherical codes from strongly regular graphs}

We now focus on codes from strongly regular graphs.  

The linear programming polynomial $f(t) = (t-p)(t-q)$ was already considered by 
Delsarte, Goethals and Seidel~\cite{DGS}. Recall that $-1<p<0<q<1$ and $p+q<0$ (see Lemma \ref{lm:SRGpq}). The Gegenbauer expansion of $f$ has coefficients 
\[
f_0 = \frac{npq + 1}{n}, \quad \quad f_1 = -p-q, \quad \quad f_2 = \frac{n-1}{n}. 
\]
Further, 
\[
\frac{f(1)}{f_0} = \frac{n(1-p)(1-q)}{npq+1} = v.
\]
So if $f_1 = -(p+q) \geq 0$, then $v$ is the maximal possible cardinality for codes with inner products in the interval $[p,q]$ (i.e., among $[-1,p)$-avoiding $q$-codes). Note that the converse statement is also true; i.e., every maximal $[-1,p)$-avoiding $q$-code is an embedding of a SRG~\cite{DGS}. 

We now prove the universal optimality of the appropriate embedding (that one with $p+q \leq 0)$ among the $[-1,p)$-avoiding codes. Recall that Lemma~\ref{lm:SRGpq} says that this assumption is rather mild.

\begin{thm}
A code obtained from a strongly regular graph with parameters $(v,k,\lambda,\mu)$ and $p+q \leq 0$ is universally optimal among the codes with scalar products in $[p,1]$.  
Namely, for every absolutely monotone $h$, if $p$ corresponds to an edge, then
\[
E_h(C) \geq \mathcal{E}_h(n, v ; [-1,p)) = v k h(p) + v (v-k-1) h(q)
\]
for $[-1,p)$-avoiding $C$ with $|C| = v$. Otherwise, $p$ corresponds to a non-edge and 
\[
E_h(C) \geq \mathcal{E}_h(n, v ; [-1,p)) = v (v-k-1) h(p) + v k h(q)
\]
for $[-1,p)$-avoiding $C$ with $|C| = v$.
\end{thm}

\begin{proof}
Consider a polynomial $g(t) = at^2 + bt + c$ with
\[
a = \frac{h(p) - h(q)}{(q-p)^2} + \frac{h'(q)}{q-p}, \quad \quad b = \frac{2q(h(q) - h(p))}{(q-p)^2} - \frac{(q+p)h'(q)}{q-p}, 
\]
\[
c = \frac{pqh'(q)}{q-p} + \frac {q^2h(p) + p(p-2q)h(q)}{(q-p)^2}.
\]
We have $q+p \leq 0$, $h(q) \geq h(p)$ and $h'(q) \geq \frac{h(q) - h(p)}{q-p}$, hence $a,b \geq 0$.

Also $g(p) = h(p)$, $g(q) = h(q)$, and $g'(q) = h'(q)$, so $g$ is the Hermite interpolant for $h$ with respect to the multi-set $\{p,q,q\}$.
Thus 
\[
h(t)-g(t)=\frac{h^{(3)}(\xi)}{3!}(t-p)(t-q)^2, \qquad p\leq t<1,
\]
which implies $g(t) \leq h(t)$ for $p \leq t < 1$.

The Gegenbauer expansion of $g(t)$ is given by 
\[
g_0 = \frac{a}{n} + c, \quad \quad g_1 = b, \quad \quad g_2 = \frac{(n-1)a}{n},
\]
so $g_1,g_2 \geq 0$. It is straightforward to check that the lower bound for energy is 
\[
v^2 g_0 - vg(1) = v k h(p) + v (v-k-1) h(q)
\]
or
\[
v^2 g_0 - vg(1) = v (v-k-1) h(p) + v k h(q)
\]
depending on the graph structure.
\end{proof}

\section{Appendix}

{\bf A1.} We list the details (polynomials and their Gegenbauer expansions) for the proofs of Theorems \ref{47104-codes-t1}--\ref{11178-codes-t1} and \ref{2816-codes-t1}--\ref{4320-codes-t4}. In all cases it is straightforward that the condition (A1) is satisfied and the explicit Gegenbauer coefficients show that (A2) is also satisfied. 

{\bf Theorem \ref{47104-codes-t1}}, case $T=\left(-1/15,1/5\right)$: apply Theorem \ref{general-lp-max-codes} with 
\[ f(t)= \left(t + \frac{3}{5}\right)^2 \left(t+\frac{1}{3}\right)^2 \left(t+\frac{1}{15}\right) \left(t-\frac{1}{5}\right) \left(t-\frac{7}{15}\right)=\sum_{i=0}^7 f_iP_i^{(23)}(t), \]
\[  f_0=\frac{256}{5821875}, \ f_1=\frac{10496}{12234375}, \ 
f_2=\frac{4894208}{506503125}, f_3=\frac{75615232}{1137796875}, \ 
f_4=\frac{1182368}{4708125}, \] 
\[ f_5=\frac{736}{1305}, \ 
f_6=\frac{86944}{121365}, \ f_7=\frac{416}{899}, \ f(1)=\frac{524288}{253125}. \]

{\bf Theorem \ref{47104-codes-t1}}, case $T=\left(-3/5,-1/3\right)$: apply Theorem \ref{general-lp-max-codes} with the polynomial $P_7$ from 
Theorem \ref{47104-en-t1}, additionally noting that $f(1)=1048576/1265625$.

{\bf Theorem \ref{93150-codes-t1}}: apply Theorem \ref{general-lp-max-codes} with 
\[ f(t)=t(t+1)\left(t + \frac{1}{2}\right)\left(t+\frac{1}{4}\right)
\left(t-\frac{1}{4}\right)\left(t-\frac{1}{2}\right)=\sum_{i=0}^6 f_iP_i^{(23)}(t), \] 
\[ f_0=\frac{1}{66240}, \ f_1=\frac{1}{2880}, \ f_2=\frac{671}{192096}, \ 
   f_3=\frac{33}{1160}, \ f_4=\frac{187}{1395}, \]
   \[ f_5=\frac{176}{261}, \  f_6=\frac{4576}{8091}, \ f(1)=\frac{45}{32}. \]
The same result can be obtained by $(t+1)f(t)$.  

{\bf Theorem \ref{11178-codes-t1}}, case $T=\left(-1/2,-1/5\right)$: apply Theorem \ref{general-lp-max-codes} with
\[ f(t)= \left(t + \frac{1}{2}\right)\left(t + \frac{1}{5}\right)
\left(t-\frac{1}{10}\right)^2\left(t-\frac{2}{5}\right)=\sum_{i=0}^5 f_iP_i^{(23)}(t), \]
\[ f_0=\frac{9}{115000}, \ f_1=\frac{37}{45000}, \ f_2=\frac{2101}{103500}, \ f_3=\frac{3663}{36250}, \ f_4=\frac{88}{1125}, \ f_5=\frac{176}{261}, 
f(1)=\frac{2187}{2500}. \] 

{\bf Theorem \ref{11178-codes-t1}}, case $T=\left(-1/5,1/10\right)$: apply Theorem \ref{general-lp-max-codes} with
\[ f(t)= \left(t + \frac{1}{2}\right)^2\left(t + \frac{1}{5}\right)
\left(t-\frac{1}{10}\right)\left(t-\frac{2}{5}\right)=\sum_{i=0}^5 f_iP_i^{(23)}(t), \]
\[ f_0=\frac{3}{23000}, \ f_1=\frac{91}{45000}, \ f_2=\frac{2827}{103500}, \ 
f_3=\frac{7491}{36250}, \ f_4=\frac{616}{1125}, \ f_5=\frac{176}{261}, 
\ f(1)=\frac{729}{500}. \]
                            
{\bf Theorem \ref{2816-codes-t1}}, case $T=(-1,-1/3)$: apply Theorem \ref{general-lp-max-codes} with
\[ f(t)= t^2(t+1)\left(t + \frac{1}{3}\right)\left(t-\frac{1}{3}\right)=\sum_{i=0}^5 f_iP_i^{(22)}(t), \]
\[ f_0=\frac{1}{1584}, \ f_1=\frac{19}{1872}, \ f_2=\frac{49}{429}, \ f_3=\frac{31}{144}, \ f_4=\frac{161}{208}, \ f_5=\frac{69}{104}, \ f(1)=\frac{16}{9}. \] 

{\bf Theorem \ref{2816-codes-t1}}, case $T=(-1/3,0)$: apply Theorem \ref{general-lp-max-codes} with
\[ f(t)= t(t+1)^2\left(t + \frac{1}{3}\right)\left(t-\frac{1}{3}\right)=\sum_{i=0}^5 f_iP_i^{(22)}(t), \]
\[ f_0=\frac{1}{792}, \ f_1=\frac{5}{208}, \ f_2=\frac{98}{429}, \ f_3=\frac{157}{144}, \ f_4=\frac{161}{104}, \ f_5=\frac{69}{104}, \ f(1)=\frac{32}{9}. \] 

{\bf Theorem \ref{2025-codes-t2}}: apply Theorem \ref{general-lp-max-codes} with
\[ f(t)= (t+1)\left(t + \frac{4}{11}\right)\left(t + \frac{1}{44}\right)\left(t-\frac{7}{22}\right)=\sum_{i=0}^4 f_iP_i^{(22)}(t), \]
\[ f_0=\frac{5}{5324}, \ f_1=\frac{691}{42592}, \ f_2=\frac{48699}{276848}, \ f_3=\frac{329}{352}, \ f_4=\frac{161}{208}, \ f(1)=\frac{10125}{5324}. \] 

{\bf Theorem \ref{4320-codes-t4}}: apply Theorem \ref{general-lp-max-codes} with
\[ f(t)= t(t+1)\left(t + \frac{1}{2}\right) \left(t + \frac{1}{4}\right)
\left(t-\frac{1}{4}\right)\left(t-\frac{1}{2}\right)=\sum_{i=0}^6
f_iP_i^{(16)}(t), \]
\[ f_0=\frac{1}{3072}, \ f_1=\frac{1}{192}, \ f_2=\frac{255}{11264}, \ f_3=\frac{125}{1056}, \ f_4=\frac{85}{384}, \ f_5=\frac{51}{88}, \ f_6=\frac{323}{704}, \ f(1)= \frac{45}{32}. \] 
The same result can be obtained by $(t+1)f(t)$.   

\medskip

{\bf A2.} We present the polynomials used in Theorems 
\ref{47104-des-t2}--\ref{4320-des-t1}. In all cases the condition (B2) of Theorem \ref{general-lp-des} follows simply by $\deg(f) \leq \tau$
and the condition (B1) can be easily verified from the explicit formula for $f$. 

{\bf Theorem \ref{47104-des-t2}}: apply Theorem \ref{general-lp-des} with 
\[ f(t)= \left(t + \frac{3}{5}\right) \left(t+\frac{1}{3}\right) \left(t+\frac{1}{15}\right) \left(t-\frac{1}{5}\right) \left(t-\frac{7}{15}\right)^2, \]
\[ f(1)= \frac{131072}{253125}, \ f_0=\frac{64}{5821875}, \ \frac{f(1)}{f_0}=47104. \] 

{\bf Theorem \ref{93150-des-t1}}, case $T=\left(-1/2,-1/4\right) \cup \left(1/4,1/2\right)$: apply Theorem \ref{general-lp-des} with
\[ f(t)= (t + 1)t^2\left(t + \frac{1}{2}\right) \left(t+\frac{1}{4}\right)\left(t-\frac{1}{4}\right) \left(t-\frac{1}{2}\right), \]
\[ f(1)=\frac{45}{32}, \ f_0=\frac{1}{66240}, \ \frac{f(1)}{f_0}=93150. \] 

{\bf Theorem \ref{93150-des-t1}}, case $\left(-1/4,0\right) \cup \left(1/4,1/2\right)$: apply Theorem \ref{general-lp-des} with 
\[ f(t)= (t + 1)t\left(t + \frac{1}{2}\right)^2 \left(t+\frac{1}{4}\right)\left(t-\frac{1}{4}\right) \left(t-\frac{1}{2}\right). \] 
\[ f(1)=\frac{135}{64}, \ f_0=\frac{1}{44160}, \ \frac{f(1)}{f_0}=93150. \] 

{\bf Theorem \ref{93150-des-t1}}, case $\left(-1/2,-1/4\right) \cup \left(0,1/4\right)$: apply Theorem \ref{general-lp-des} with 
\[ f(t)= (t + 1)t\left(t + \frac{1}{2}\right) \left(t+\frac{1}{4}\right)\left(t-\frac{1}{4}\right) \left(t-\frac{1}{2}\right)^2, \]
\[ f(1)=\frac{45}{64}, \ f_0=\frac{1}{132480}, \ \frac{f(1)}{f_0}=93150. \] 

{\bf Theorem \ref{11178-des-t1}}: apply Theorem \ref{general-lp-des} with
\[ f(t)= \left(t + \frac{1}{2}\right)\left(t+\frac{1}{5}\right)
  \left(t - \frac{1}{10}\right)\left(t-\frac{2}{5}\right),  \]
\[ f(1)=\frac{243}{250}, \ f_0=\frac{1}{11500}, \ \frac{f(1)}{f_0}=11178. \] 

{\bf Theorem \ref{2816-des-t1}}, case $T=(0,1/3)$: apply Theorem \ref{general-lp-des} with 
\[ f(t)= t(t+1)\left(t + \frac{1}{3}\right)^2 \left(t-\frac{1}{3}\right), \]
\[ f(1)= \frac{64}{27}, \ f_0=\frac{1}{1188}, \ \frac{f(1)}{f_0}=2816. \] 

{\bf Theorem \ref{2816-des-t1}}, case $T=(-1/3,0)$: apply Theorem \ref{general-lp-des} with 
\[ f(t)= t(t+1)\left(t + \frac{1}{3}\right) \left(t-\frac{1}{3}\right)^2, \]
\[ f(1)=\frac{32}{27}, \ f_0=\frac{1}{2376}, \ \frac{f(1)}{f_0}=2816. \] 

{\bf Theorem \ref{2025-des-t1}}, case $(-1/44,7/22)$: apply Theorem \ref{general-lp-des} with 
\[ f(t)= \left(t + \frac{4}{11}\right)^2 \left(t+\frac{1}{44}\right)\left(t-\frac{7}{22}\right), \]
\[ f(1)=\frac{151875}{117128}, \ f_0=\frac{75}{117128}, \ \frac{f(1)}{f_0}=2025. \] 

{\bf Theorem \ref{2025-des-t1}}, case $(-4/11,-1/44)$: apply Theorem \ref{general-lp-des} with 
\[ f(t)= \left(t + \frac{4}{11}\right) \left(t+\frac{1}{44}\right)\left(t-\frac{7}{22}\right)^2, \]
\[ f(1)=\frac{151875}{234256}, \ f_0=\frac{75}{234256}, \ \frac{f(1)}{f_0}=2025. \] 

{\bf Theorem \ref{4320-des-t1}}, case $T=\left(-1/2,-1/4\right) \cup \left(1/4,1/2\right)$: apply Theorem \ref{general-lp-des} with
\[ f(t)= (t + 1)t^2\left(t + \frac{1}{2}\right) \left(t+\frac{1}{4}\right)\left(t-\frac{1}{4}\right) \left(t-\frac{1}{2}\right),  \]
\[ f(1)=\frac{45}{32}, \ f_0=\frac{1}{3072}, \ \frac{f(1)}{f_0}=4320. \] 

{\bf Theorem \ref{4320-des-t1}}, case $\left(-1/4,0\right) \cup \left(1/4,1/2\right)$: apply Theorem \ref{general-lp-des} with \[ f(t)= (t + 1)t\left(t + \frac{1}{2}\right)^2 \left(t+\frac{1}{4}\right)\left(t-\frac{1}{4}\right) \left(t-\frac{1}{2}\right). \]
\[ f(1)=\frac{135}{64}, \ f_0=\frac{1}{2048}, \ \frac{f(1)}{f_0}=4320. \] 

{\bf Theorem \ref{4320-des-t1}}, case $\left(-1/2,-1/4\right) \cup \left(0,1/4\right)$: apply Theorem \ref{general-lp-des} with
\[ f(t)= (t + 1)t\left(t + \frac{1}{2}\right) \left(t+\frac{1}{4}\right)\left(t-\frac{1}{4}\right) \left(t-\frac{1}{2}\right)^2, \]
\[ f(1)=\frac{45}{64}, \ f_0=\frac{1}{6144}, \ \frac{f(1)}{f_0}=4320. \] 

\medskip

{\bf A3.} We present the polynomials from Section 7.

{\bf Theorem \ref{47104-en-t1}:} Here we verify the positive definiteness of the partial products for the remaining three choices of $T$ in \eqref{T_47104}. 

When $T=(-1/3,-1/15)$ the corresponding interpolation set is
$$\mathcal{I}=\left\{-\frac{3}{5},-\frac{3}{5}, -\frac{1}{3},-\frac{1}{15},\frac{1}{5}, \frac{1}{5},\frac{7}{15},\frac{7}{15}\right\}.$$
There are three partial products with factors $(t-c)$, $c>0$ for which we compute
\[ P_5(t)=\left(t + \frac{3}{5}\right)^2 \left(t+\frac{1}{3}\right) \left(t+\frac{1}{15}\right) \left(t-\frac{1}{5}\right)=\sum_{i=0}^5 f_iP_i^{(23)}(t), \]
with
\[f_0=\frac{728}{129375}, \ f_1=\frac{344}{5625}, \ f_2=\frac{2552}{8625}, \ f_3=\frac{127336}{163125}, \ f_4=\frac{1232}{1125}, \ f_5=\frac{176}{261} ;\]

\[ P_6(t)=\left(t + \frac{3}{5}\right)^2 \left(t+\frac{1}{3}\right) \left(t+\frac{1}{15}\right) \left(t-\frac{1}{5}\right)^2=\sum_{i=0}^6 f_iP_i^{(23)}(t), \]
where
\[f_0=\frac{992}{646875}, \ f_1=\frac{32}{1875}, \ f_2=\frac{581152}{6753375}, \ f_3=\frac{217888}{815625}, \ f_4=\frac{915904}{1569375},\] 
\[f_5=\frac{352}{435},f_6=\frac{4576}{8091} ;\]
and
\[ P_7(t)=\left(t + \frac{3}{5}\right)^2 \left(t+\frac{1}{3}\right) \left(t+\frac{1}{15}\right) \left(t-\frac{1}{5}\right)^2\left( t-\frac{7}{15}\right)=\sum_{i=0}^7 f_iP_i^{(23)}(t), \]
for which the coefficients are: 
\[f_0=\frac{256}{9703125}, \ f_1=\frac{3328}{7340625}, \ f_2=\frac{2962432}{506503125}, \ f_3=\frac{13274624}{379265625}, \ f_4=\frac{450208}{4708125}, \]
\[f_5=\frac{13408}{58725}, \ f_6=\frac{50336}{121365}, \ f_7= \frac{416}{899} .\]

For the interval $T=(-1/15,1/5)$ the corresponding interpolation set is
$$\mathcal{I}=\left\{-\frac{3}{5},-\frac{3}{5}, -\frac{1}{3},-\frac{1}{3},-\frac{1}{15}, \frac{1}{5},\frac{7}{15},\frac{7}{15}\right\}.$$
Now there are two partial products with factors $(t-c)$, $c>0$, $P_6$ and $P_7$. We have 
\[ P_6(t)=\left(t + \frac{3}{5}\right)^2 \left(t+\frac{1}{3}\right)^2 \left(t+\frac{1}{15}\right) \left(t-\frac{1}{5}\right)=\sum_{i=0}^6 f_iP_i^{(23)}(t), \]
where
\[
f_0=\frac{352}{77625}, \ f_1=\frac{4192}{84375}, \ f_2=\frac{8234336}{33766875}, \ f_3=\frac{1672352}{2446875}, \ f_4=\frac{1832512}{1569375},\] 
\[f_5=\frac{4576}{3915}, \ f_6=\frac{4576}{8091} \]
for $P_6$, while $P_7$ coincides with the polynomial $f$ from the case 
$T=(-1/15,1/5)$ of Theorem \ref{47104-codes-t1}.

Finally, for $T=(1/5,7/15)$ the interpolation set is
\[
\mathcal{I} = \left\{-\frac{3}{5},-\frac{3}{5}, -\frac{1}{3},-\frac{1}{3},-\frac{1}{15}, -\frac{1}{15},\frac{1}{5},\frac{7}{15}\right\},
\]
and we have only one partial product to verify. 

\[ P_7(t)=\left(t + \frac{3}{5}\right)^2 \left(t+\frac{1}{3}\right)^2 \left(t+\frac{1}{15}\right)^2 \left(t-\frac{1}{5}\right)=\sum_{i=0}^7 f_iP_i^{(23)}(t), \]
for which the coefficients are: 
\[f_0=\frac{14336}{5821875}, \ f_1=\frac{1004032}{36703125}, \ f_2=\frac{23589632}{168834375}, \ f_3=\frac{490358528}{1137796875}, \ f_4=\frac{6857312}{7846875}, \]
\[f_5=\frac{69728}{58725}, \ f_6=\frac{4576}{4495}, \ f_7= \frac{416}{899} .\]
\hfill $\Box$

{\bf Theorem \ref{93150-en-t1}:} First, consider $T=(-1/2,-1/4)\cup(0,1/4)$. Then 
\[ 
\mathcal{I} = \left \{-1,-1,-\frac{1}{2},-\frac{1}{4},0,\frac{1}{4},\frac{1}{2},\frac{1}{2} \right\}.
\]
In this case only the partial products $P_6$ and $P_7$ have factors $(t-c)$, $c>0$. We find, respectively, that
\[ P_6(t)=\left(t +1\right)^2 \left(t+\frac{1}{2}\right) \left(t+\frac{1}{4}\right) t \left(t-\frac{1}{4}\right)=\sum_{i=0}^6 f_iP_i^{(23)}(t), \]
with
\[
f_0=\frac{467}{82800}, \ f_1=\frac{59}{900}, \ f_2=\frac{4169}{12006}, \ f_3=\frac{12309}{11600}, \ f_4=\frac{13211}{6975}, \ f_5=\frac{440}{261}, \ f_6=\frac{4576}{8091} ;\]
and that
\[ P_7(t)=\left(t +1\right)^2 \left(t+\frac{1}{2}\right) \left(t+\frac{1}{4}\right) t \left(t-\frac{1}{4}\right)\left(t-\frac{1}{2}\right)=\sum_{i=0}^7 f_iP_i^{(23)}(t), \]
has coefficients \[f_0=\frac{1}{33120}, \ f_1=\frac{67}{104400}, \ f_2=\frac{671}{96048}, \ f_3=\frac{36069}{719200}, \ f_4=\frac{374}{1395}, \] \[f_5=\frac{233}{261}, \ f_6=\frac{9152}{8091}, \ f_7= \frac{416}{899} .\]

Next, let $T=(-1/2,-1/4)\cup(1/4,1/2)$. The interpolation set is 
\[ 
\mathcal{I} = \left\{-1,-1,-\frac{1}{2},-\frac{1}{4},0,0,\frac{1}{4},\frac{1}{2} \right\}.
\]
In this case we need to consider only one partial product
\[ P_7(t)=\left(t +1\right)^2 \left(t+\frac{1}{2}\right) \left(t+\frac{1}{4}\right) t^2 \left(t-\frac{1}{4}\right)=\sum_{i=0}^7 f_iP_i^{(23)}(t), \]
with coefficients \[f_0=\frac{59}{20700}, \ f_1=\frac{1163}{34800}, \ f_2=\frac{17347}{96048}, \ f_3=\frac{26103}{44950}, \ f_4=\frac{16951}{13950}, \] \[f_5=\frac{151}{87}, \ f_6=\frac{11440}{8091}, \ f_7= \frac{416}{899} .\]

Finally, let $T=(-1/4,0)\cup(1/4,1/2)$. The interpolation set is 
\[ 
\mathcal{I} = \left\{-1,-1,-\frac{1}{2},-\frac{1}{2},-\frac{1}{4},0,\frac{1}{4},\frac{1}{2} \right\}.
\]
Again, only one partial product has a factor $(t-c)$, $c>0$.
\[ P_7(t)=\left(t +1\right)^2 \left(t+\frac{1}{2}\right)^2 \left(t+\frac{1}{4}\right) t \left(t-\frac{1}{4}\right)=\sum_{i=0}^7 f_iP_i^{(23)}(t), \]
and its coefficients are all positive
\[f_0=\frac{313}{55200}, \ f_1=\frac{6911}{104400}, \ f_2=\frac{11341}{32016}, \ f_3=\frac{799227}{719200}, \ f_4=\frac{5027}{2325}, \] 
\[ f_5=\frac{673}{261}, \ f_6=\frac{4576}{2697}, \ f_7= \frac{416}{899} .\]
\hfill $\Box$

{\bf Theorem \ref{11178-en-t1}:} We shall consider the three cases identified in \eqref{T_11178}. 

First, when $T=(-1/2,-1/5)$ we have that the interpolation set is 
\[ 
\mathcal{I} = \left\{-\frac{1}{2},-\frac{1}{5},\frac{1}{10},\frac{1}{10},\frac{2}{5},\frac{2}{5} \right\}.
\]
There are three partial products we need to consider, $P_3$, $P_4$, and $P_5$. We have
\[ P_3(t)=\left(t+ \frac{1}{2}\right) \left(t+\frac{1}{5}\right) \left(t-\frac{1}{10}\right)=\sum_{i=0}^3 f_iP_i^{(23)}(t), \]
with
\[f_0=\frac{37}{2300}, \ f_1=\frac{3}{20}, \ f_2=\frac{66}{115}, \ f_3=\frac{22}{25}.\]
We also find that
\[ P_4(t)=\left(t+ \frac{1}{2}\right) \left(t+\frac{1}{5}\right) \left(t-\frac{1}{10}\right)^2=\sum_{i=0}^4 f_iP_i^{(23)}(t), \]
where
\[f_0=\frac{113}{23000}, \ f_1=\frac{47}{1000}, \ f_2=\frac{1903}{10350}, \ f_3=\frac{11}{25}, \ f_4=\frac{176}{225},\]
as well as $P_5$ which is the same as $f$ from the case $T=(-1/2,-1/5)$
in Theorem \ref{11178-codes-t1}.

Next, we consider $T=(-1/5,1/10)$. The interpolation set becomes 
\[ 
\mathcal{I} = \left \{-\frac{1}{2},-\frac{1}{2},-\frac{1}{5},\frac{1}{10},\frac{2}{5},\frac{2}{5} \right \}.
\]
Two partial products have all their factors of the type $t-c$, $c>0$. We find
\[ 
P_4(t)=\left(t+ \frac{1}{2}\right)^2 \left(t+\frac{1}{5}\right) \left(t-\frac{1}{10}\right)=\sum_{i=0}^4 f_iP_i^{(23)}(t), 
\]
where
\[
f_0=\frac{67}{4600}, \ f_1=\frac{137}{1000}, \ f_2=\frac{5467}{10350}, \ f_3=\frac{121}{125}, \ f_4=\frac{176}{225},
\]
and $P_5$ which is the same as $f$ from the case $T=(-1/5,1/10)$
in Theorem \ref{11178-codes-t1}.

And finally, when $T=(1/10,2/5)$. The interpolation set becomes 
\[ 
\mathcal{I} = \left \{-\frac{1}{2},-\frac{1}{2},-\frac{1}{5},-\frac{1}{5},\frac{1}{10},\frac{2}{5} \right \}.
\]
We need to verify only
\[ 
P_5(t)=\left(t+ \frac{1}{2}\right)^2 \left(t+\frac{1}{5}\right)^2 \left(t-\frac{1}{10}\right)=\sum_{i=0}^5 f_iP_i^{(23)}(t), 
\]
where
\[
f_0=\frac{51}{5750}, \ f_1=\frac{379}{4500}, \ f_2=\frac{35629}{103500}, \ f_3=\frac{5709}{7250}, \ f_4=\frac{1144}{1125}, \ f_5=\frac{176}{261}.
\]

\hfill $\Box$

{\bf Theorem \ref{48600-en-t1}:} In the first case $T=(-13/23,-7/23)\cup(-1/23,5/23)$ the interpolating set is 
\[ 
\mathcal{I} = \left\{-\frac{13}{23},-\frac{7}{23}, -\frac{1}{23}, \frac{5}{23}, \frac{11}{23}, \frac{11}{23} \right\}.
\]
We have to verify positive definiteness of two partial products:
\[ 
P_4(t)=\left(t+ \frac{13}{23}\right) \left(t+\frac{7}{23}\right) \left(t+\frac{1}{23}\right)\left(t-\frac{5}{23}\right)=\sum_{i=0}^4 f_iP_i^{(23)}(t), 
\]
where
\[
f_0=\frac{28576}{6996025}, \ f_1=\frac{13792}{304175}, \ f_2=\frac{24464}{109503}, \ f_3=\frac{352}{575}, \ f_4=\frac{176}{225},
\]
and 
\[ 
P_5(t)=\left(t+ \frac{13}{23}\right) \left(t+\frac{7}{23}\right) \left(t+\frac{1}{23}\right)\left(t-\frac{5}{23}\right)\left(t-\frac{11}{23}\right)=\sum_{i=0}^5 f_iP_i^{(23)}(t), 
\]
with
\[
f_0=\frac{576}{32181715}, \ f_1=\frac{3424}{12592845}, \ f_2=\frac{11440}{2518569}, \ f_3=\frac{1584}{76705}, \ f_4=\frac{176}{1035}, \ f_5=\frac{176}{261}.
\]

The second case is $T=(-13/23,-7/23)\cup(5/23,11/23)$ the interpolating set is 
\[
\mathcal{I} = \left\{-\frac{13}{23},-\frac{7}{23}, -\frac{1}{23}, -\frac{1}{23}, \frac{5}{23}, \frac{11}{23} \right\},
\]
in which case we need to consider only the last partial products
\[ 
P_5(t)=\left(t+ \frac{13}{23}\right) \left(t+\frac{7}{23}\right) \left(t+\frac{1}{23}\right)^2\left(t-\frac{5}{23}\right)=\sum_{i=0}^5 f_iP_i^{(23)}(t), 
\]
with
\[
f_0=\frac{345792}{160908575}, \ f_1=\frac{1506656}{62964225}, \ f_2=\frac{305008}{2518569}, \ f_3=\frac{130416}{383525}, \ f_4=\frac{2992}{5175}, \ f_5=\frac{176}{261}.
\]

The last case is when 
$T=(-7/23,-1/23)\cup(5/23,11/23)$ with interpolating set 
\[
\mathcal{I} = \left \{-\frac{13}{23},-\frac{13}{23}, -\frac{7}{23}, -\frac{1}{23}, \frac{5}{23}, \frac{11}{23} \right\}.
\]
Again, we need to examine only the last partial product
\[ 
P_5(t)=\left(t+ \frac{13}{23}\right)^2 \left(t+\frac{7}{23}\right) \left(t+\frac{1}{23}\right)\left(t-\frac{5}{23}\right)=\sum_{i=0}^5 f_iP_i^{(23)}(t), 
\]
in which case the coefficients are again positive
\[
f_0=\frac{688704}{160908575}, \ f_1=\frac{2996192}{62964225}, \ f_2=\frac{598576}{2518569}, \ f_3=\frac{252912}{383525}, \ f_4=\frac{5104}{5175}, \ f_5=\frac{176}{261}.
\]
This concludes the proof of the theorem.
\hfill $\Box$
\medskip

{\bf Theorem \ref{2816-en-t1}:} For $T=(-1,-1/3)$ the interpolation set is 
\[ \mathcal{I}=\left\{-1, -\frac{1}{3},0,0,\frac{1}{3},\frac{1}{3}\right\},\]
and the only partial product we need to consider is $P_5$, which is, in fact, the polynomial $f$ from the case $T=(-1,-1/3)$ in Theorem \ref{2816-codes-t1}. 

When $T=(-1/3,0)$ the interpolation set becomes
\[ \mathcal{I}=\left\{-1, -1, -\frac{1}{3},0,\frac{1}{3},\frac{1}{3}\right\}.\]
Thus, only the last partial product needs verification, and it is 
the polynomial $f$ from the case $T=(-1/3,0)$ in Theorem \ref{2816-codes-t1}. 

The last interval in \eqref{T_2816} is $T=(0,1/3)$ with an interpolation set
\[ \mathcal{I}=\left\{-1, -1, -\frac{1}{3},-\frac{1}{3}, 0,\frac{1}{3}\right\}.\]
In this case all factors are of type $(t+c)$, $c\geq 0$, so all partial products have positive Gegenbauer coefficients and the proof is complete.
\hfill $\Box$

\medskip
{\bf Theorem \ref{4320-en-t1}:} For the first case, $T=(-1/2,-1/4)\cup (0,1/4)$, the interpolation set is 
\[ \mathcal{I}=\left\{-1,-1,-\frac{1}{2},-\frac{1}{4},0,\frac{1}{4},\frac{1}{2},\frac{1}{2} \right\}.\]
There are two partial products with factors $(t-c)$, $c>0$ and we find that 
\[ 
P_7(t)=\left(t+ 1\right)^2 \left(t+\frac{1}{2}\right)\left(t+\frac{1}{4}\right)t \left(t-\frac{1}{4}\right)\left(t-\frac{1}{2}\right)=\sum_{i=0}^7 f_iP_i^{(16)}(t), 
\]
with
\[
f_0=\frac{1}{1536}, \ f_1=\frac{17}{2112}, \ f_2=\frac{255}{5632}, \ f_3=\frac{755}{4224}, \ f_4=\frac{85}{192}, \
f_5=\frac{15861}{18304}, \ f_6=\frac{323}{352}, \ 
f_7=\frac{1615}{4576},\]
and 
\[ 
P_6(t)=\left(t+ 1\right)^2 \left(t+\frac{1}{2}\right)\left(t+\frac{1}{4}\right)t \left(t-\frac{1}{4}\right)=\sum_{i=0}^6 f_iP_i^{(16)}(t), 
\]
where
\[
f_0=\frac{23}{1536}, \ f_1=\frac{25}{192}, \ f_2=\frac{2949}{5632}, \ f_3=\frac{2605}{2112}, \ f_4=\frac{697}{384}, \ 
f_5=\frac{255}{176}, \ f_6=\frac{323}{704}.\]

When $T=(-1/2,-1/4)\cup (1/4,1/2)$, the interpolation set is 
\[ \mathcal{I}=\left\{-1,-1,-\frac{1}{2},-\frac{1}{4},0,0,\frac{1}{4},\frac{1}{2} \right\}.\]
The last partial product is
\[ 
P_7(t)=\left(t+ 1\right)^2 \left(t+\frac{1}{2}\right)\left(t+\frac{1}{4}\right)t^2 \left(t-\frac{1}{4}\right)=\sum_{i=0}^7 f_iP_i^{(16)}(t), 
\]
with
\[
f_0=\frac{25}{3072}, \ f_1=\frac{103}{1408}, \ f_2=\frac{3459}{11264}, \ f_3=\frac{35}{44}, \ f_4=\frac{1037}{768}, \ 
f_5=\frac{29121}{18304}, \] \[ f_6=\frac{1615}{1408}, \ 
f_7=\frac{1615}{4576},\]
which establishes the positive definiteness of $H_7(t;h,T)$ in this case as well.

To complete the proof of the theorem we consider $T=(-1/4,0)\cup (1/4,1/2)$
with interpolation set 
\[ \mathcal{I}= \left\{-1,-1,-\frac{1}{2},-\frac{1}{2},-\frac{1}{4},0,\frac{1}{4},\frac{1}{2} \right\}.\]
We verify that the last partial product is positive-definite:
\[ 
P_7(t)=\left(t+ 1\right)^2 \left(t+\frac{1}{2}\right)^2\left(t+\frac{1}{4}\right)t \left(t-\frac{1}{4}\right)=\sum_{i=0}^7 f_iP_i^{(16)}(t), 
\]
where
\[
f_0=\frac{1}{64}, \ f_1=\frac{73}{528}, \ f_2=\frac{801}{1408}, \ f_3=\frac{5965}{4224}, \ f_4=\frac{289}{128}, \ 
f_5=\frac{42381}{18304}, \ f_6=\frac{969}{704}, \
f_7=\frac{1615}{4576},\]
ending the proof.
\hfill $\Box$

\bigskip

{\bf Acknowledgment.} We thank the honoree for enthusiastically promoting the subject of minimal energy and for his influence and inspiration throughout our careers.

\bigskip
{\bf Funding.} The research of the first two authors is supported by Bulgarian NSF grant KP-06-N72/6-2023. The research of the third author is supported, in part, by the Lilly Endowment.

\end{document}